\documentclass[10pt]{amsart}
\usepackage{amscd}
\usepackage{amsmath}
\usepackage{amsfonts}
\usepackage{amssymb}
\usepackage{amsthm}
\usepackage{enumerate}
\usepackage[T1]{fontenc}
\usepackage{hyperref}
\usepackage{mathrsfs}
\usepackage{stackrel}
\usepackage[all]{xy}
\usepackage{yhmath}
\usepackage{comment}

\newtheorem{thm}{Theorem}[section]
\newtheorem{lem}[thm]{Lemma}
\newtheorem{defn}[thm]{Definition}
\newtheorem{prop}[thm]{Proposition}
\newtheorem{cor}[thm]{Corollary}

\newtheorem{ques}[thm]{Question}

\DeclareMathOperator{\Sp}{Sp}

\newcommand{\A}{\mathcal{A}}
\newcommand{\B}{\mathcal{B}}
\newcommand{\C}{\mathcal{C}}
\newcommand{\D}{\mathcal{D}}

\newcommand{\F}{\mathcal{F}}

\newcommand{\I}{\mathcal{I}}

\renewcommand{\L}{\mathcal{L}}

\newcommand{\M}{\mathcal{M}}

\newcommand{\N}{\mathcal{N}}

\renewcommand{\O}{\mathcal{O}}

\newcommand{\R}{\mathcal{R}}

\newcommand{\T}{\mathcal{T}}

\newcommand{\h}[1]{\widehat{#1}}

\newcommand{\w}[1]{\wideparen{#1}}

\begin{document}
\title{Holonomic $\w{\D}$-modules on rigid analytic spaces}
\author{Andreas Bode}
\maketitle
\begin{abstract}
We adapt Caro's notion of overholonomicity to give a definition of holonomic $\w{\D}$-modules on rigid analytic spaces. We prove stability under five of the six operations (both inverse image functors, duality, and both direct image functors for projective morphisms), as well as base change results. Up to the open problem of stability under tensor products, we obtain an analogue of the usual six-functor formalism for holonomic $\D$-modules.
\end{abstract}
\tableofcontents
\section{Introduction}
In this paper, which can be regarded as a sequel to \cite{Bode6Op}, we introduce the notion of \emph{holonomic} $\w{\D}$-modules on smooth rigid analytic varieties and study behaviour under the six functors defined in loc. cit. 

Recall that if $X$ is a smooth complex algebraic variety, a coherent $\D_X$-module $\M$ is called holonomic if it is of minimal dimension, i.e. either $\M=0$ or $\mathrm{dim} \M=\mathrm{dim}X$ (which is the smallest possible dimension due to Bernstein's inequality). Here, the dimension of a coherent $\D$-module can be defined as the dimension of the associated characteristic variety, or equivalently, taking a homological viewpoint, in terms of the homological grade, i.e. the vanishing of Ext groups.

The most crucial properties of holonomic $\D$-modules include the stability of holonomicity under all six $\D$-module theoretic operations: the direct image $f_+$, the inverse image $f^+$, their shriek versions $f_!$ and $f^!$, tensor products over $\O$, and duality. Moreover, holonomic $\D$-module are `constructible' out of integrable connections ($\O$-coherent $\D$-modules), which allows one to formulate a Riemann--Hilbert correspondence for holonomic $\D$-modules with regular singularities.

In fact, (regular) holonomic $\D$-modules offer a key example of (a version of) a six-functor formalism, which allows us to treat them as `coefficient objects' for de Rham cohomology. Apart from some differences in convention, it is worth noting that the approach here is much more explicit (and as a consequence, more minimalist) than for abstract six-functor formalisms: one defines six functors explicitly on some larger class of $\D$-modules, and then exhibits holonomic modules as a smaller subcategory on which these functors are actually well-behaved. 

Let $K$ now be a complete nonarchimedean field of mixed characteristic $(0, p)$. We are concerned with an analogous theory in rigid analytic geometry, considering the completed sheaf $\w{\D}_X$ from \cite{DcapOne} on a smooth rigid analytic $K$-variety $X$. The role of coherent $\D$-modules is here played by coadmissible $\w{\D}$-modules. In \cite{Bode6Op}, we have introduced a derived categorical framework for complete bornological $\w{\D}$-modules, which allowed us to define natural analogues of the six functors in this setting. The analogue of $\mathrm{D}^b_{\mathrm{coh}}(\D)$ in this setting is the category of $\C$-complexes $\mathrm{D}_{\C}(\w{\D})$. We now turn to the question of holonomicity and stability under these operations.

Note that \cite{DcapThree} already developed a dimension theory for coadmissible $\w{\D}$-modules, but the corresponding class of modules of minimal dimension, which we call weakly holonomic $\w{\D}$-modules, contains many undesired pathologies. For instance, it is not true that the extraordinary inverse image of a weakly holonomic $\w{\D}$-module is again weakly holonomic (as \cite{DcapThree} discusses an example with infinite-dimensional fibres -- we will return to this example in section 6).

To introduce a better notion of holonomicity, we follow Caro's approach (e.g. \cite{Caro04}, \cite{Caro09}) in his work on overholonomic $F-\mathscr{D}^\dag$-modules, and essentially make the stability properties part of the definition. We say that a $\C$-complex is $0$-holonomic if it remains a $\C$-complex after pullback along smooth morphisms and taking local cohomology with support in any divisor. This is the direct analogue of what Caro calls `overcoherent' modules in \cite{Caro04} (also, `$0$-overholonomic' in the terminology of \cite{Caro09}). The definition of holonomicity (`overholonomicity' in \cite{Caro09}) is then given inductively to ensure in addition stability under the duality functor (see Definition \ref{defnhol} for the precise definition). 

In other words: A $\C$-complex is called holonomic if it remains a $\C$-complex under any finite succession of duality, local cohomology and inverse image operations for smooth morphisms.

In adopting Caro's definitions, we have decided to change the terminology and drop the `over', due to the absence of any overconvergence phenomena in our theory as well as the current absence of a workable definition of holonomicity. We remark that both in the complex algebraic setting and in Caro's setting, overholonomicity can be shown to be equivalent to the usual notion of holonomicity in terms of minimal dimension. In contrast, it is worth pointing out that in our setting, holonomicity is not equivalent to weak holonomicity, as we are ruling out the pathologies mentioned in \cite{DcapThree}, for instance.

We remark that we define what it means for an object in $\mathrm{D}_{\C}(\w{\D})$ to be holonomic, rather than for a coadmissible $\w{\D}$-module, i.e. we define a certain triangulated subcategory $\mathrm{D}_{\mathrm{hol}}(\w{\D})$ of $\mathrm{D}(\w{\D})$. Of course, a coadmissible $\w{\D}$-module is then called holonomic if it can be regarded as a holonomic complex, concentrated in degree zero. We expect that a bounded $\C$-complex $\M^\bullet\in \mathrm{D}_{\C}(\w{\D})$ is holonomic if and only if each of its cohomology groups is holonomic (compare Lemma \ref{holcohom}). We refer to section 6 for a discussion of this and several related open questions, in particular whether holonomic implies weakly holonomic.
 
The following result ensures that our definition is reasonable.
\begin{thm}[{Theorem \ref{ICishol}}]
Let $X$ be a smooth rigid analytic $K$-variety. Then any integrable connection on $X$ is a holonomic $\w{\D}_X$-module.
\end{thm} 
We prove the following stability properties.
\begin{thm}[{Theorem \ref{sixffhol}}]
Let $f: X\to Y$ be a morphism of smooth rigid analytic $K$-varieties.
\begin{enumerate}[(i)]
\item $f^!$ and $f^+$ send $\mathrm{D}_{\mathrm{hol}}(\w{\D}_Y)$ to $\mathrm{D}_{\mathrm{hol}}(\w{\D}_X)$.
\item If $f$ is projective, then $f_+\cong f_!$ sends $\mathrm{D}_{\mathrm{hol}}(\w{\D}_X)$ to $\mathrm{D}_{\mathrm{hol}}(\w{\D}_Y)$.
\item The duality functor $\mathbb{D}_X$ sends $\mathrm{D}_{\mathrm{hol}}(\w{\D}_X)$ to $\mathrm{D}_{\mathrm{hol}}(\w{\D}_X)^{\mathrm{op}}$.
\end{enumerate}
\end{thm}
For direct images, some properness or projectiveness condition is necessary -- see the end of section 5 for a detailed discussion and counterexamples even in the case of partially proper morphisms. We express the hope that stability under tensor products can be proved following again Caro--Tsuzuki's approach \cite{CaroTsuzuki}, by first showing that holonomic $\w{\D}$-modules are in some sense `constructible' out of integrable connections. 

We are pleased to say that once the bornological theory has been set up properly, all arguments regarding holonomic modules turn out to be purely formal: while section 3 still involves a bit of non-archimedean analysis, the proofs in the following sections essentially only rely on a handful of general properties of $\mathrm{D}_\C(\w{\D})$ and standard diagram chases.

We briefly mention how this work relates to alternative formulations. Instead of working with sheaves, both \cite{dR} and \cite{Soor} suggest another approach by essentially associating to any (smooth) affinoid $U$ a suitable $\infty$-category $\mathrm{D}(\w{\D}(U))$, satisfying descent -- for example, one could take $\mathrm{D}(\mathrm{Mod}_{LH(\h{\B}c_K)}(\w{\D}_X(U)))$, or in the case of \cite{dR}, a suitable category of modules on the analytic de Rham stack. As $\C$-complexes are obtained by glueing objects in $\mathrm{D}^b_{\mathrm{coh}}(\D_n)$, it is not surprising that our category $\mathrm{D}_{\C}(\w{\D}_U)$ (or rather, its $\infty$-categorical enhancement) naturally embeds into such categories, see \cite[Corollary 1.5, Theorem 1.6]{Soor}. Given the formal nature of much of our argument in this paper, one may expect that an analogous definition yields an equivalent notion of holonomicity in these different settings -- it essentially remains to check that our operations `do the same thing' on $\C$-complexes as the functors in other frameworks, i.e. if $\theta: \mathrm{D}_{\C}(\w{\D}_U)\to \mathrm{D}(\w{\D}(U))$ is the embedding, and e.g. $f: V\to U$ is a morphism of smooth affinoids, then $f^!\M^\bullet$ is a $\C$-complex if and only if $f^!\theta(\M^\bullet)$ is a $\C$-complex, and in this case $f^!\theta(\M^\bullet)\cong \theta(f^!\M^\bullet)$. Since our definition of holonomicity only involves duality, inverse image along a smooth morphism and local cohomology (equivalently, $j_+j^!$ for some divisor complement $j: U\setminus D\to U$), this seems very plausible.

We give a brief overview over the structure of this paper.

In section 2, we recall the main results from \cite{Bode6Op}.

In section 3, we give further results on $\C$-complexes, involving localisation triangles for closed subvarieties (excision), exterior tensor products, and the commutativity of $\mathbb{D}$ with smooth pullback.

In sections 4 and 5, we introduce $0$-holonomic and holonomic $\w{\D}$-modules, respectively, and prove their stability properties. These sections are purely formal: We prove a projective base change theorem for $0$-holonomic complexes, which reduces most of our arguments to straightforward diagram chases.

We then discuss further examples and open questions in section 6.

\subsection*{Conventions and notation}
Throughout, $K$ denotes a complete nonarchimedean field of mixed characteristic $(0, p)$. We let $R=\{x\in K| \ |x|\leq 1\}$ denote its ring of integers, $\mathfrak{m}\subset R$ the maximal ideal. We pick $\pi\in \mathfrak{m}$ with $\pi\neq 0$ as a pseudo-uniformizer.

All our rigid analytic $K$-varieties are assumed to be quasi-separated and quasi-paracompact. When using expressions like $\mathrm{dim}X$, we tacitly assume that $X$ is pure of dimension $\mathrm{dim}X$.

\section{$\mathcal{C}$-complexes}
We begin by recalling the main results from \cite{Bode6Op} -- we refer to loc. cit. for a more detailed discussion.
\begin{defn}
A \textbf{bornological $K$-vector space} is a $K$-vector space $V$ equipped with a collection $\mathcal{B}$ of `bounded' subsets of $V$ such that
\begin{enumerate}[(i)]
\item if $B\in \mathcal{B}$ and $B'\subseteq B$, then $B'\in \mathcal{B}$.
\item if $v\in V$, then $\{v\}\in \mathcal{B}$.
\item $\mathcal{B}$ is closed under taking finite unions.
\item if $B\in \mathcal{B}$ and $\lambda\in K$ is non-zero, then $\lambda \cdot B\in \mathcal{B}$.
\item if $B\in \mathcal{B}$, then $R\cdot B$, the $R$-submodule spanned by $B$, is an element of $\mathcal{B}$.
\end{enumerate} 
\end{defn}
The category $\B c_K$ has as objects bornological $K$-vector spaces, and morphisms are precisely those $K$-linear maps which are bounded in the sense that they send bounded subsets to bounded subsets.

We say that $V$ is \textbf{complete} if its bornology is generated by complete $R$-submodules, i.e. any bounded subset of $V$ is contained in a bounded $R$-submodule which is $\pi$-adically complete. In particular, any complete bornological vector space can be written as a colimit of Banach spaces (both in the category $\B c_K$ of bornological vector spaces and in the category $\h{\B}c_K$ of complete bornological vector spaces).

We remark that Fr\'echet $K$-spaces can be viewed as complete bornological vector spaces in a natural way, by declaring those subsets to be bounded which are bounded with respect to each semi-norm in a family of semi-norms defining the Fr\'echet topology.

The category $\h{\B}c_K$ is a quasi-abelian category in the sense of Schneiders \cite{Schneiders}, admitting a natural theory of a (unbounded) derived category. As a quasi-abelian category, $\h{\B}c_K$ admits a canonical embedding $I$ into its left heart (its `abelian envelope') $LH(\h{\B}c_K)$, an abelian category such that $I$ induces an equivalence
\begin{equation*}
\mathrm{D}(\h{\B}c_K)\cong \mathrm{D}(LH(\h{\B}c_K))
\end{equation*} 
of triangulated categories.

Moreover, $\h{\B}c_K$ can be equipped with a closed symmetric monoidal structure using the projective completed tensor product (extending the usual completed tensor product for Banach spaces).

It was noted in \cite[subsections 4.2, 6.1]{Bode6Op} that it is not straightforward to develop a theory of complete bornological sheaves, as $\h{\B}c_K$ is not an \emph{elementary} quasi-abelian category. However, the quasi-abelian category $\mathrm{Ind}(\mathrm{Ban}_K)$ of Ind-Banach spaces is elementary, and there is a natural functor
\begin{equation*}
	\mathrm{diss}: \h{\B}c_K\to \mathrm{Ind}(\mathrm{Ban}_K)
\end{equation*}
inducing an equivalence of closed symmetric monoidal categories between the respective left hearts. The tensor product on $LH(\h{\B}c_K)$, which can be obtained by deriving the completed tensor product on $\h{\B}c_K$ or equivalently by extending the completed tensor product on Banach spaces to $\mathrm{Ind}(\mathrm{Ban}_K)$, will be denoted by $\widetilde{\otimes}_K$.

It is now possible to develop a derived theory of sheaves on a site $X$ with the category
\begin{equation*}
	\mathrm{D}(\mathrm{Shv}(X, LH(\h{\B}c_K)))\cong \mathrm{D}(\mathrm{Shv}(X, LH(\mathrm{Ind}(\mathrm{Ban}_K))))\cong \mathrm{D}(\mathrm{Shv}(X, \mathrm{Ind}(\mathrm{Ban}_K))).
\end{equation*} 

We can now consider categories of modules over monoid objects. For instance, if $X$ is a rigid analytic $K$-variety, then the natural Banach structure of $\O_X(U)$ for any affinoid subdomain $U$ turns $\O_X$ into monoid in $\mathrm{Shv}(X, LH(\h{\B}c_K))$, and we can study the category of $\O_X$-module objects in $\mathrm{Shv}(X, LH(\h{\B}c_K))$, which contains as a full subcategory all coherent $\O_X$-modules (in turn equipped with their natural Banach structures).

Regarding $\h{\B}c_K$ as a full subcategory of $LH(\h{\B}c_K)$ and of $\mathrm{Ind}(\mathrm{Ban}_K)$, we will use the following shorthand: a monoid $\R \in \mathrm{Shv}(X, LH(\h{\B}c_K))$ is called a sheaf of complete bornological $K$-algebras if $\R(U)\in \h{\B}c_K$ for each $U$. In the same way, we can talk of sheaves of complete bornological $\R$-modules. 

Now let $X$ be a smooth rigid analytic $K$-variety. We can define a sheaf $\w{\D}_X$ of complete bornological vector spaces, which locally is given as follows: suppose that $X$ is affinoid and that $X=\Sp A$ admits a local coordinate system, i.e. there exist $x_1, \dots, x_m\in A$ and $\partial_1, \dots, \partial_m\in \T_X(X)$, a free generating set of $\T_X(X)$ over $A$, satisfying
\begin{equation*}
[\partial_i, \partial_j]=0, \ \partial_i(x_j)=\delta_{ij}
\end{equation*}
for all $i, j$.

Evidently, such a coordinate system corresponds to an \'etale morphism $X\to \Sp K\langle X_1, \hdots, X_m\rangle$, with $x_i\in A$ the image of $X_i$, and $\partial_i$ the corresponding partial derivative.

Let $\A\subseteq A$ be an $R$-subalgebra of topologically finite type which spans $A$, i.e. an admissible affine formal model of $A$. After rescaling the $x_i$ and $\partial_i$, we can suppose that $\partial_i(\A)\subseteq \A$. Let $\L\subseteq \T_X(X)$ denote the $\A$-submodule generated by the $\partial_i$. By assumption, $\L$ is an $(R, \A)$-Lie lattice in the Lie--Rinehart algebra $\T_X(X)$, so that we can form the completed enveloping algebra 
\begin{equation*}
\h{U_\A(\L)}\otimes_R K,
\end{equation*}
and similarly $\h{U_\A(\pi^n\L)}\otimes_R K$ for any $n\geq 0$ -- see \cite{Rinehart} and \cite{DcapOne} for a general discussion of enveloping algebras of Lie--Rinehart algebras.

Then we have explicitly
\begin{align*}
\w{\D}_X(X)&=\left\{ \sum_{\alpha\in \mathbb{N}^m} f_\alpha \partial^\alpha: \ \pi^{-n|\alpha|}f_\alpha\to 0 \ \mathrm{as} \ |\alpha|\to \infty \  \forall n\right\}\\
&=\varprojlim_{n\geq 0} \left(\h{U_\A(\pi^n\L)}\otimes_R K\right).
\end{align*}
This is a Fr\'echet $K$-algebra (in fact, a Fr\'echet--Stein algebra in the sense of \cite{ST}), so we can regard it as a complete bornological $K$-algebra. 

If $X$ is any smooth rigid analytic $K$-variety, assigning each admissible open affinoid subspace $U$ (with free tangent sheaf, say) the corresponding algebra $\w{\D}_U(U)$ defines a monoid in $\mathrm{Shv}(X, LH(\h{\B}c_K))$:

\begin{lem}[{\cite[Theorem 6.4, Theorem 3.24]{Bode6Op}}]
Let $X$ be a smooth rigid analytic $K$-variety. Then $\w{\D}_X$ is a sheaf of complete bornological $K$-algebras. The category $\mathrm{Mod}_{\mathrm{Shv}(X, LH(\h{\B}c_K))}(\w{\D}_X)$ is a Grothendieck abelian category admitting flat resolutions.
\end{lem}
We generally write $\mathrm{D}(\w{\D}_X)=\mathrm{D}(\mathrm{Mod}_{\mathrm{Shv}(X, LH(\h{\B}c_K))}(\w{\D}_X))$ for the derived category, similarly for other monoids. We will work throughout in this setting, so the term `$\w{\D}_X$-module' will always refer to an object of $\mathrm{Mod}_{\mathrm{Shv}(X, LH(\h{\B}c_K))}(\w{\D}_X)$.
\begin{lem}[{\cite[Theorem 6.11]{Bode6Op}}]
	\label{sidechanging}
Let $\Omega_X=\Omega_X^{\wedge \mathrm{dim}X}$ be the sheaf of top differentials on $X$. Then $\Omega_X\widetilde{\otimes}_{\O_X}-$ yields an equivalence of categories between left $\w{\D}_X$-modules and right $\w{\D}_X$-modules.
\end{lem}
We can now define six operations in analogy with the six functors for $\D$-modules on smooth algebraic varieties:

There is a tensor product bifunctor
\begin{align*}
-\widetilde{\otimes}^\mathbb{L}_{\O_X}-:& \mathrm{D}(\w{\D}_X)\times \mathrm{D}(\w{\D}_X)\to \mathrm{D}(\w{\D}_X)\\
&(\M^\bullet, \N^\bullet)\mapsto \M^\bullet\widetilde{\otimes}^\mathbb{L}_{\O_X} \N^\bullet
\end{align*}
and a duality functor
\begin{align*}
\mathbb{D}:& \mathrm{D}(\w{\D}_X)\to \mathrm{D}(\w{\D}_X)^{\mathrm{op}}\\
&\M^\bullet\mapsto \mathrm{R}\mathcal{H}om_{\w{\D}_X}(\M^\bullet, \w{\D}_X)\widetilde{\otimes}^{\mathbb{L}}_{\O_X}\Omega_X^{\otimes -1} [\mathrm{dim} X].
\end{align*}
Given a morphism $f: X\to Y$ between smooth rigid analytic $K$-varieties, we define the transfer bimodule
\begin{equation*}
\w{\D}_{X\to Y}=f^*\w{\D}_Y=\O_X\widetilde{\otimes}_{f^{-1}\O_Y}f^{-1} \w{\D}_Y,
\end{equation*}
a $(\w{\D}_X, f^{-1}\w{\D}_Y)$-bimodule, and its side-changed version
\begin{equation*}
\w{\D}_{Y\leftarrow X}=\Omega_X\widetilde{\otimes}_{\O_X}\w{\D}_{X\to Y}\widetilde{\otimes}_{f^{-1}\O_Y} f^{-1}\Omega_Y^{\otimes-1},
\end{equation*}
an $(f^{-1}\w{\D}_Y, \w{\D}_X)$-bimodule.

We then set the extraordinary inverse image functor
\begin{align*}
f^!: &\mathrm{D}(\w{\D}_Y)\to \mathrm{D}(\w{\D}_X)\\
&\M^\bullet\mapsto \w{\D}_{X\to Y}\widetilde{\otimes}^\mathbb{L}_{f^{-1}\w{\D}_Y} f^{-1}\M^\bullet [\mathrm{dim}X-\mathrm{dim}Y],
\end{align*}
and the inverse image functor $f^+=\mathbb{D}_Xf^!\mathbb{D}_Y$.

The direct image functor is given by
\begin{align*}
f_+: & \mathrm{D}(\w{\D}_X)\to \mathrm{D}(\w{\D}_Y)\\
& \M^\bullet \mapsto \mathrm{R}f_*(\w{\D}_{Y\leftarrow X}\widetilde{\otimes}^\mathbb{L}_{\w{\D}_X} \M^\bullet),
\end{align*}
and the extraordinary direct image functor is $f_!=\mathbb{D}_Yf_+\mathbb{D}_Y$. We refer to \cite[section 7]{Bode6Op} for further results, e.g. concerning the composition of these functors.

We generally follow the $\D$-module theoretic conventions for our six operations, which differ somewhat from what one might expect from an abstract six-functor formalism. For example, the degree shift in $f^!$ is motivated by the Riemann--Hilbert correspondence, which prevents $f^!$ from being monoidal. This introduces various degree shifts in our formulas, which might get corrected e.g. by incorporating the shift into the side-changing functor from Lemma \ref{sidechanging} instead -- see e.g. \cite{Scholze6FF} for a discussion of $\D$-module operations from the viewpoint of abstract six-functor formalisms. We have opted to keep our framework as closely to the classical theory of $\D$-modules as possible. 

We now turn to the analogue of $\mathrm{D}^b_\mathrm{coh}(\D)$ in this theory. Ardakov--Wadsley introduced in \cite{DcapOne} the category $\C_X$ of coadmissible $\w{\D}_X$-modules as the analogue of coherent modules. Our notion of $\C$-complexes can be regarded as a straightforward generalization to the derived setting.
Suppose again that $X=\Sp A$ is affinoid (for simplicity, we can even assume it has a local coordinate system). As mentioned above, $\w{\D}_X(X)$ is a Fr\'echet--Stein algebra, so we can consider coadmissible $\w{\D}_X(X)$-modules as in \cite{ST}. We say that an abstract sheaf of $\w{\D}_X$-modules is coadmissible if it is obtained from a coadmissible $\w{\D}_X(X)$-module via completed localisation -- see \cite[section 8]{DcapOne} and \cite[subsection 6.2]{Bode6Op} for details.

As in the above, assume we have chosen an affine formal model $\A\subseteq A$ and an $(R, \A)$-Lie lattice $\L\subseteq \T_X(X)$. Above, we have written $\w{\D}_X(X)$ as an inverse limit of Banach algebras. We can carry out a similar description on the sheaf-theoretic level, as long as we are ready to vary the site: there exist subsites $X_n$ for each $n\geq 0$, satisfying the following properties:
\begin{enumerate}[(i)]
\item If an affinoid subdomain $Y$ is in $X_n$, then it is in $X_m$ for $m\geq n$. Any covering in $X_n$ also defines a covering in $X_m$ for $m\geq n$.
\item Let $Y$ be an affinoid subdomain of $X$. Then $Y$ is in $X_n$ for sufficiently large $n$, and any finite affinoid covering of $Y$ is a covering in $X_n$ for sufficiently large $n$.
\item If $Y=\Sp B$ is an affinoid subdomain in $X_n$, then $B$ contains an admissible affine formal model $\mathcal{B}$, containing the image $\mathcal{A}$ under restriction, which is stable under the $\pi^n\mathcal{L}$-action. In particular, $\B\otimes_\A \pi^n\L$ defnes an $(R, \B)$-Lie lattice inside $\T_X(Y)$.
\item The assignment
\begin{equation*}
Y\mapsto \mathcal{D}_n(Y):=\h{U_{\mathcal{B}}(\mathcal{B}\otimes_\mathcal{A} \pi^n\mathcal{L})}\otimes_RK,
\end{equation*}
which is independent of the choice of $\B$, defines a sheaf of $K$-algebras on $X_n$ which has vanishing higher Cech cohomology with respect to any $X_n$-covering.
\end{enumerate}
The site $X_n$ and the sheaf $\D_n$ obviously depend on our choice of $\L$ and in general can only be defined on affinoid spaces. 

If $\M$ is a coherent $\D_n$-module, then $\M(U)$ carries a natural Banach structure for any affinoid subdomain $U$ in $X_n$. This allows us to regard coherent $\D_n$-modules as complete bornological $\D_n$-modules. The following result gives a similar description of coadmissible $\w{\D}_X$-modules from the complete bornological viewpoint.
\begin{lem}[{\cite[Theorem 6.4, Lemma 6.6]{Bode6Op}}]
There is an exact, fully faithful functor from coadmissible $\w{\D}_X$-modules to $\mathrm{Mod}_{\mathrm{Shv}(X, LH(\h{\B}c_K))}(\w{\D}_X)$. If $X$ is an affinoid with free tangent sheaf, then a complete bornological $\w{\D}_X$-module $\M$ is coadmissible if and only if $\M_n:=\D_n\widetilde{\otimes}_{\w{\D}_X}\M$ is a coherent $\D_n$-module for each $n$ and the natural morphism $\M\to \varprojlim \M_n$ is an isomorphism.
\end{lem}

\begin{defn}
Let $X$ be an affinoid with free tangent sheaf. An object $\M^\bullet$ in $\mathrm{D}(\w{\D}_X)$ is a \textbf{$\C$-complex} if
\begin{enumerate}[(i)]
\item $\M^\bullet_n:=\D_n\widetilde{\otimes}^\mathbb{L}_{\w{\D}_X} \M^\bullet\in \mathrm{D}^b_{\mathrm{coh}}(\D_n)$ for each $n$, and
\item the natural morphism $\mathrm{H}^i(\M^\bullet)\to \varprojlim \mathrm{H}^i(\M^\bullet_n)$ is an isomorphism for each $i$.
\end{enumerate}
\end{defn}
Equivalently (\cite[Proposition 8.5]{Bode6Op}), $\M^\bullet$ is a $\C$-complex if and only if each of its cohomology groups is a coadmissible $\w{\D}_X$-module, and $\M^\bullet_n$ is bounded for each $n$. It is straightforward to glue this definition to obtain a subcategory $\mathrm{D}_\C\subseteq \mathrm{D}(\w{\D}_X)$ of $\C$-complexes for any smooth rigid analytic $K$-variety $X$.

The behaviour of $\C$-complexes under our six operations is analogous to that of coherence in the classical setting.
\begin{thm}[{\cite[Theorem 1.3]{Bode6Op}}]
\label{Ccomplexprops}
\leavevmode
\begin{enumerate}[(i)]
\item Let $i: X\to Y$ be a closed embedding of smooth rigid analytic $K$-varieties. Then $i_+$ and $i^!$ induce equivalences of categories between $\mathrm{D}_\C(\w{D}_X)$ and those $\C$-complexes in $\mathrm{D}_\C(\w{D}_Y)$ whose cohomology groups are supported on $X$.
\item If $f: X\to Y$ is a smooth morphism, then $f^!$ preserves $\C$-complexes.
\item If $f: X\to Y$ is a projective morphism, then $f_+$ preserves $\C$-complexes.
\item $\mathbb{D}_X$ preserves $\mathrm{D}_\C(\w{\D}_X)$ and there is a natural isomorphism $\M^\bullet\cong \mathbb{D}^2\M^\bullet$ for any $\C$-complex $\M^\bullet$.  
\end{enumerate}
\end{thm}

Points (iii) and (iv) were initially proved only over discretely valued fields. See \cite{Bodepoincare} and \cite[Proposition 6.5]{Bodegldim} for the generalization to arbitrary $K$.

A typical strategy in proving statements about $\C$-complexes is the following: We say that a smooth affinoid $\Sp A$ is \textbf{D-regular} if it admits an admissible affine formal model $\A$ and an $(R, \A)$-Lie lattice $\L$ which is free as an $\A$-module such that $\h{U_{\A}(\pi^n\L)}\otimes_RK$ is Auslander regular for all $n\geq 0$. The main result of \cite{Bodegldim} is that every smooth rigid analytic $K$-variety can be covered by D-regular affinoids (\cite[Theorem 1.1]{Bodegldim}). Locally, we can thus find suitable lattices such that any object in $\mathrm{D}_{\mathrm{coh}}^b(\D_n)$ is perfect, reducing many arguments to the case of a free $\D_n$-module of rank one, and a study of the limit as $n$ goes to infinity.

We remark that in some instances, behaviour of the direct image functor can be quite subtle, forcing us to consider the following restrictions on the morphisms we consider.

\begin{defn}
	A separated morphism $f: X\to Y$ between smooth rigid analytic $K$-varieties is called \textbf{managable} if there exists an admissible covering $(Y_i)$ of $Y$ such that for each $i$, the following is satisfied:
	\begin{enumerate}[(i)]
		\item $Y_i$ is a D-regular affinoid.
		\item $f^{-1}Y_i$ has only countably many connected components.
		\item The functor $\mathrm{pr}_{i, *}$ for the projection morphism $\mathrm{pr}_i: f^{-1}Y_i\times Y_i\to Y_i$ has finite cohomological dimension.
	\end{enumerate}
\end{defn}

\begin{prop}
	\label{compandadj}
	Let $f: X\to Y$, $g: Y\to Z$ be morphisms of smooth rigid analytic $K$-varieties.
	\begin{enumerate}[(i)]
		\item (See \cite[Proposition 7.5]{Bode6Op})There is a natural isomorphism
		\begin{equation*}
			(gf)^!\M^\bullet\cong f^!g^!\M^\bullet
		\end{equation*}
		for any $\M^\bullet\in \mathrm{D}(\w{\D}_Z)$.
		\item (See \cite[Lemma 7.8, Corollary 9.8]{Bode6Op}) If $f$ is managable or $g$ is smooth, then there is a natural isomorphism
		\begin{equation*}
			(gf)_+\M^\bullet\cong g_+f_+\M^\bullet
		\end{equation*}
		for any $\M^\bullet\in \mathrm{D}_{\C}(\w{\D}_X)$.
		\item (See \cite[Theorem 6.23]{Bodegldim}) If $f$ is managable, there is a natural isomorphism
		\begin{equation*}
			\mathrm{R}\mathcal{H}om_{\w{\D}_Y}(f_!\M^\bullet, \N^\bullet)\cong \mathrm{R}f_*\mathcal{H}om_{\w{\D}_X}(\M^\bullet, f^!\N^\bullet)
		\end{equation*}
		for any $\M^\bullet\in \mathrm{D}_{\C}(\w{\D}_X)$ with $f_!\M^\bullet\in \mathrm{D}_{\C}(\w{\D}_Y)$ and any $\N^\bullet\in \mathrm{D}_{\C}(\w{\D}_Y)$ with $f^!\N^\bullet\in \mathrm{D}_{\C}(\w{\D}_X)$.
		\item (See \cite[Theorem 6.23]{Bodegldim}) If $f$ is managable, there is a natural isomorphism
		\begin{equation*}
			\mathrm{R}f_*\mathrm{R}\mathcal{H}om_{\w{\D}_X}(f^+\N^\bullet, \M^\bullet)\cong \mathrm{R}\mathcal{H}om_{\w{\D}_Y}(\N^\bullet, f_+\M^\bullet)
		\end{equation*}
		for any $\M^\bullet\in \mathrm{D}_{\C}(\w{\D}_X)$ with $f_+\M^\bullet\in \mathrm{D}_{\C}(\w{\D}_Y)$ and any $\N^\bullet\in \mathrm{D}_{\C}(\w{\D}_Y)$ with $f^+\N^\bullet\in \mathrm{D}_{\C}(\w{\D}_X)$.
	\end{enumerate}
\end{prop}

Lastly, we point out that the Kashiwara equivalence in Theorem \ref{Ccomplexprops}.(i) implies the following.
\begin{prop}
	\label{Kashiwaratest}
	Let $i: Z\to X$ be a Zariski closed embedding of smooth rigid analytic $K$-varieties. If $\M^\bullet\in \mathrm{D}(\w{\D}_Z)$ such that $i_+\M^\bullet\in \mathrm{D}_\C(\w{\D}_X)$, then $\M^\bullet\in \mathrm{D}_\C(\w{\D}_Z)$.
\end{prop}
\begin{proof}
	As the question is local, we can assume that $Z$ and $X$ are both affinoid admitting a local coordinate system, with $Z$ being given as the vanishing set of some coordinates. Suppose $X$ has local coordinate system $x_1, \hdots, x_m, \partial_1, \hdots, \partial_m$, and $Z=\{x_{s+1}=\hdots =x_m=0\}$. The natural choice of lattices then induces compatible sites $Z_n$ and $X_n$ respectively. 
	
	Let $\I=(x_{s+1}, \hdots, x_m)$ denote the vanishing ideal of $Z$. Using the equivalence between left and right modules (and the fact that this equivalence respects coadmissibility, \cite[Theorem 3.5]{DcapTwo}), we can work with right $\w{\D}$-modules instead. We let $i_+^r$ and $i^!_r$ denote the corresponding functors for right modules.
	
	Note that in this case, 
	\begin{equation*}
		i_+^r\M\cong i_*(\M\widetilde{\otimes}_K K\langle \partial_{s+1},\hdots, \partial_m\rangle)\cong i_*\M\widetilde{\otimes}_K K\langle \partial_{s+1}, \hdots, \partial_m\rangle
	\end{equation*}
	 for any $\M\in \mathrm{Mod}_{\mathrm{Shv}(Z, LH(\h{\B}c_K))}(\w{\D}_Z)$ by \cite[Lemma 3.30]{Bode6Op}. In particular, $i_+^r$ is conservative.
	
	Suppose that $\M$ is a right $\w{\D}_Z$-module such that $i_+^r\M$ is right coadmissible, which is then supported on $Z$ by construction. By \cite[Theorem 9.2, Lemma 9.3]{Bode6Op}, 
	\begin{equation*}
		i_r^!i^r_+\M\cong i^{-1}((i_+^r\M)[\I])
	\end{equation*}
	is a coadmissible right $\w{\D}_Z$-module. Here, $(i_+^r\M)[\I]$ denotes the $\I$-torsion of $i_+^r\M$.
	
	We thus have a natural morphism $\M\to i_r^!i_+^r\M$ which becomes an isomorphism after applying $i_+^r$. But now $i_+^r: LH(\w{\D}_Z^\mathrm{op})\to LH(\w{\D}_X^\mathrm{op})$ is conservative, so $\M\cong i^!_ri_+^r\M$ is coadmissible.
	
	It thus follows from exactness of $i_+^r$ that if $\M^\bullet\in \mathrm{D}(\w{\D}_Z^\mathrm{op})$ such that $i_+^r\M^\bullet$ is a $\C$-complex, then $\M^\bullet$ has coadmissible cohomology. 
	
	It now remains to show that for each $n$, the coherent $\D_{Z_n}$-module
	\begin{equation*}
		\mathrm{H}^j(\M^\bullet\widetilde{\otimes}_{\w{\D}_Z}^{\mathbb{L}}\D_{Z_n})\cong \mathrm{H}^j(\M^\bullet)\widetilde{\otimes}_{\w{\D}_Z}\D_{Z_n}
	\end{equation*}
	vanishes for all but finitely many $j$ (see \cite[Corollary 5.38]{Bode6Op} for the fact that coadmissible modules are $-\widetilde{\otimes}_{\w{\D}_Z}\D_{Z_n}$-acyclic).
	
	But
	\begin{equation*}
		i_*(\mathrm{H}^j(\M^\bullet)\widetilde{\otimes}_{\w{\D}_Z}\D_{Z_n}\widetilde{\otimes}_{\D_{Z_n}} i^{-1} \frac{\D_{X_n}}{\mathcal{I}\D_{X_n}})\cong (i_+^r\mathrm{H}^j(\M^\bullet))\widetilde{\otimes}_{\w{\D}_X} \D_{X_n},
	\end{equation*}
	which vanishes for all but finitely many $j$ by the definition of a $\C$-complex. The result thus follows again from $i_*(-\widetilde{\otimes}_{\D_{Z_n}} i^{-1} \frac{\D_{X_n}}{\mathcal{I}\D_{X_n}})$ being conservative -- compare also \cite[Theorem 4.10]{DcapTwo}.
\end{proof}

\section{Further properties of $\mathcal{C}$-complexes}

In this section, we collect supplementary results on the theory of $\mathcal{C}$-complexes. More precisely, we discuss localisation triangles (excision) in subsection 3.1 (see Proposition \ref{triangle}), exterior tensor products and their behaviour under the six operations in subsection 3.2, and the commutation of the duality functor with smooth pullback (Theorem \ref{pullbackdual}).

\subsection{Localisation triangles}
\begin{lem}
\label{injisflabby}
Let $X$ be a smooth rigid analytic $K$-variety. If $\F$ is an injective object in $\mathrm{Mod}_{\mathrm{Shv}(X, LH(\h{\B}c_K))}(\w{\D}_X)$, then it is flabby in the sense that for any admissible open $U\subseteq X$, the restriction morphism $\F(X)\to \F(U)$ is an epimorphism.
\end{lem}
\begin{proof} 
Let $U\subseteq X$ be an admissible open subset and write $j: U\to X$ for the embedding. Let $j_!\w{\D}_U$ be the $\w{\D}_X$-module obtained by sheafifying
\begin{equation*}
V\mapsto \begin{cases}\w{\D}_U(V) \ \text{if} \ V\subseteq U\\
0 \ \text{otherwise}.
\end{cases}
\end{equation*} 
For any $\w{\D}_X$-module $\F$, we then have
\begin{equation*}
\mathrm{Hom}_{\w{\D}_X}(j_!\w{\D}_U, \F)\cong \F(U)
\end{equation*}
via the natural morphism. Since $j_!\w{\D}_U\to \w{\D}_X$ is a monomorphism in $\mathrm{Shv}(X, LH(\h{\B}c_K))$ (and hence in $\mathrm{Mod}_{\mathrm{Shv}(X, LH(\h{\B}c_K))}(\w{\D}_X)$) due to the exactness of sheafification (\cite[Lemma 3.15]{Bode6Op}), this implies that if $\mathcal{F}$ is an injective module, then $\F(X)\to \F(U)$ is an epimorphism, i.e. injective modules are flabby.
\end{proof} 

If $Z$ is a Zariski closed subvariety of a smooth rigid analytic $K$-variety $X$, let $j: U\to X$ denote the complement. Since $j$ is \'etale, we have $j_+j^!\M\cong \mathrm{R}j_*(\M|_U)$ for any $\w{\D}_X$-module $\M$. The kernel of the natural morphism
\begin{equation*}
	\M\to j_*(\M|_U)
\end{equation*}
is called the sheaf of sections of $\M$ supported on $Z$, and will be denoted $\underline{\Gamma}_Z(\M)$. By construction, 
\begin{equation*}
	\underline{\Gamma}_Z: \mathrm{Mod}_{\mathrm{Shv}(X, LH(\h{\B}c_K))}(\w{\D}_X)\to \mathrm{Mod}_{\mathrm{Shv}(X, LH(\h{\B}c_K))}(\w{\D}_X)
\end{equation*}
is a left exact functor, admitting a derived functor $\mathrm{R}\underline{\Gamma}_Z: \mathrm{D}(\w{\D}_X)\to \mathrm{D}(\w{\D}_X)$ by \cite[Theorem 14.3.1]{KS}.

We call $\mathrm{R}\underline{\Gamma}_Z$ the \textbf{local cohomology functor} and $\mathrm{R}^i\underline{\Gamma}_Z(\M)$ the $i$th local cohomology group with support in $Z$, compare also \cite{Kisin}.

\begin{prop}
\label{triangle}
Let $X$ be a smooth rigid analytic $K$-variety, $Z$ a Zariski closed subvariety and $U=X\setminus Z$. Let $i: Z\to X$ and $j: U\to X$ denote the embeddings.
\begin{enumerate}[(i)]
\item If $\mathcal{M}^\bullet\in \mathrm{D}(\wideparen{\mathcal{D}}_X)$, then there is a distinguished triangle
\begin{equation*}
\mathrm{R}\underline{\Gamma}_Z(\mathcal{M}^\bullet)\to \mathcal{M}^\bullet\to j_+j^!\M^\bullet\to
\end{equation*}
in $\mathrm{D}(\w{\D}_X)$.
\item If $Z$ is smooth, then
\begin{equation*}
i^!j_+\mathcal{M}^\bullet=0
\end{equation*}
for any $\mathcal{M}^\bullet\in \mathrm{D}_\mathcal{C}(\wideparen{\mathcal{D}}_U)$ such that $j_+\M^\bullet\in \mathrm{D}_\C(\w{\D}_X)$.
\item If $Z$ is smooth and $\M^\bullet\in \mathrm{D}_\C(\w{\D}_X)$ satisfies $j_+j^!\M^\bullet\in \mathrm{D}_\C(\w{\D}_X)$, then $\mathrm{R}\underline{\Gamma}_Z(\mathcal{M}^\bullet)\cong i_+i^!(\mathcal{M}^\bullet)$, so that we obtain a distinguished triangle
\begin{equation*}
i_+i^!(\mathcal{M}^\bullet)\to \mathcal{M}^\bullet\to j_+j^!\mathcal{M}^\bullet\to
\end{equation*}
in $\mathrm{D}(\wideparen{\mathcal{D}}_X)$.
\end{enumerate}
\end{prop}
\begin{proof}
\begin{enumerate}[(i)]
\item Let $\mathcal{F}$ be an injective object in $\mathrm{Mod}_{\mathrm{Shv}(X, LH(\h{\B}c_K))}(\w{\D}_X)$. Then the sequence
\begin{equation*}
0 \to \underline{\Gamma}_Z(\F)\to \F\to j_*j^{-1}\F\to 0
\end{equation*}
is exact, as injectivity implies flabbiness by Lemma \ref{injisflabby}. If $\M^\bullet\in \mathrm{D}(\w{\D}_X)$, taking a  K-injective resolution yields the result.
\item For any $\mathcal{F}\in \mathrm{D}(\O_U)$, the natural morphism
\begin{equation*}
(i_*\O_Z)\widetilde{\otimes}^\mathbb{L}_{\O_X} \mathrm{R}j_*\mathcal{F}\to \mathrm{R}j_*(j^{-1}i_*\O_Z\widetilde{\otimes}^\mathbb{L}_{\O_U}\mathcal{F})
\end{equation*}
is an isomorphism, as we can locally replace $i_*\O_Z$ by a Koszul resolution. Note that the right hand side is $0$, as elements of the vanishing ideal of $Z$ are units in $\O_U$.

It remains to show that the natural morphism
\begin{equation*}
i_*\O_Z\widetilde{\otimes}^\mathbb{L}_{\O_X} \N^\bullet\to \mathrm{R}i_*(\O_Z\widetilde{\otimes}^\mathbb{L}_{i^{-1}\O_X}i^{-1}\N^\bullet)
\end{equation*}  
is also an isomorphism whenever $\N^\bullet=\mathrm{R}j_*\M^\bullet$ is a $\C$-complex, as then 
\begin{align*}
\mathrm{R}i_*(i^!\mathrm{R}j_*\M^\bullet)[\mathrm{dim}X-\mathrm{dim} Z]&\cong \mathrm{R}i_*(\O_Z\widetilde{\otimes}^\mathbb{L}_{i^{-1}\O_X} i^{-1}\mathrm{R}j_*\M^\bullet)\\
&=0.
\end{align*}
implies the result.

But this is precisely the projection formula from \cite[Corollary 6.16]{Bodegldim}, noting that
\begin{equation*}
	i_+^r\w{\D}_Z\cong i_*\O_Z\widetilde{\otimes}_{\O_X}\w{\D}_X.
\end{equation*}

\item As $j_+j^!\M^\bullet$ is a $\C$-complex, so is $\mathrm{R}\underline{\Gamma}_Z(\M^\bullet)$ by \cite[Proposition 8.4]{Bode6Op}. By construction (or e.g. the distinguished triangle in (i)), each cohomology group of $\mathrm{R}\underline{\Gamma}_Z(\M^\bullet)$ is supported on $Z$, so in particular
\begin{equation*}
	i_+i^!\mathrm{R}\underline{\Gamma}_Z(\M^\bullet)\cong \mathrm{R}\underline{\Gamma}_Z(\M^\bullet)
\end{equation*}. 
by Theorem \ref{Ccomplexprops}.(i). 

Applying $i^!$ to the distinguished triangle in (i), we can invoke (ii) to obtain $i^!j_+j^!\M^\bullet=0$ and the isomorphism
\begin{equation*}
	i^!\mathrm{R}\underline{\Gamma}_Z(\M^\bullet)\cong i^!\M^\bullet,
\end{equation*}
and applying $i_+$ proves the result.
\end{enumerate}
\end{proof}

The following can then be regarded as a Mayer--Vietoris result for $\w{\D}$-modules.
\begin{prop}\label{MVforDcap}
Let $Z$, $Z'$ be Zariski closed subvarieties of a smooth rigid analytic $K$-variety $X$. For any $\M^\bullet\in \mathrm{D}(\w{\D}_X)$, there is a natural distinguished triangle
\begin{equation*}
\mathrm{R}\underline{\Gamma}_{Z\cap Z'}(\M^\bullet)\to\mathrm{R}\underline{\Gamma}_{Z}(\M^\bullet)\oplus\mathrm{R}\underline{\Gamma}_{Z'}(\M^\bullet)\to\mathrm{R}\underline{\Gamma}_{Z\cup Z'}(\M^\bullet)\to.
\end{equation*}
\end{prop}
\begin{proof}
Let $j_1: U_1\to X$ and $j_2: U_2\to X$ be the complements of $Z$ and $Z'$, respectively. We write $j_{12}: U_{12}\to X$ for the complement of $Z\cup Z'$ (i.e., $U_{12}=U_1\cap U_2$), and let $j: U\to X$ be the complement of $Z\cap Z'$.

If $\mathcal{F}$ is an injective object, we have the following commutative diagram:
\begin{equation*}
\begin{xy}
\xymatrix{
&0\ar[d] & 0\ar[d] & 0\ar[d]\\
0\ar[r]&\underline{\Gamma}_{Z\cap Z'}(\mathcal{F})\ar[r] \ar[d] &\underline{\Gamma}_Z(\mathcal{F})\oplus \underline{\Gamma}_{Z'}(\mathcal{F})\ar[r] \ar[d] &\underline{\Gamma}_{Z\cup Z'}(\mathcal{F})\ar[r] \ar[d] &0\\
0\ar[r] &\mathcal{F}\ar[r] \ar[d] &\mathcal{F}\oplus \mathcal{F}\ar[r] \ar[d] &\mathcal{F} \ar[d] \ar[r] &0\\
0\ar[r] & j_*\mathcal{F}|_{U}\ar[r] \ar[d]&{j_1}_*\mathcal{F}|_{U_1}\oplus {j_2}_*\mathcal{F}|_{U_2}\ar[r] \ar[d]& {j_{12}}_*\mathcal{F}|_{U_{12}}\ar[r]\ar[d] & 0\\
&0 & 0 & 0 
}
\end{xy}
\end{equation*}
where each column is exact by Lemma \ref{injisflabby}. The middle row, induced by the diagonal mapping $\mathcal{F}\to \mathcal{F}\oplus \mathcal{F}$, is clearly exact. Note that the third row is also exact: it is left exact by the sheaf axioms, and right exact by the injectivity of $\mathcal{F}$.

Hence the first row is also exact, yielding the desired distinguished triangle.
\end{proof}
\subsection{Exterior tensor product}
Let $X$ and $Y$ be smooth rigid analytic $K$-varieties. Let $p_1: X\times Y\to X$ and $p_2: X\times Y\to Y$ denote the natural projections. 

For $\mathcal{M}^\bullet\in \mathrm{D}(\O_X)$, $\mathcal{N}^\bullet\in \mathrm{D}(\O_Y)$, we define the \textbf{exterior tensor product} by
\begin{equation*}
\mathcal{M}^\bullet\boxtimes \mathcal{N}^\bullet=\O_{X\times Y} \widetilde{\otimes}^\mathbb{L}_{p_1^{-1}\O_X\widetilde{\otimes}_Kp_2^{-1}\O_Y}(p_1^{-1}\mathcal{M}^\bullet\widetilde{\otimes}_K^{\mathbb{L}}p_2^{-1}\mathcal{N}^\bullet)\in \mathrm{D}(\O_{X\times Y}).
\end{equation*}

\begin{lem}
	\label{boxtimesOdiagonal}
	Let $X$ be a smooth rigid analytic $K$-variety and let $\Delta: X\to X\times X$ denote the diagonal embedding. There is a natural isomorphism
	\begin{equation*}
		\M_1^\bullet\widetilde{\otimes}_{\O_X}^{\mathbb{L}}\M_2^\bullet\cong \mathbb{L}\Delta^*(\M_1^\bullet\boxtimes \M_2^\bullet)
	\end{equation*}
	for any $\M_1^\bullet, \M_2^\bullet\in \mathrm{D}(\O_X)$.
\end{lem}
Here, as usual, we use the notation
\begin{equation*}
	\mathbb{L}\Delta^*(\M^\bullet)=\O_X\widetilde{\otimes}_{\Delta^{-1}\O_{X\times X}}^{\mathbb{L}}\Delta^{-1}\M^\bullet.
\end{equation*}
\begin{proof}
	The natural morphism
	\begin{equation*}
		\M_1^\bullet\boxtimes \M_2^\bullet\to \mathrm{R}\Delta_*(\M_1^\bullet\widetilde{\otimes}_{\O_{X\times X}}^{\mathbb{L}}\M_2^\bullet)
	\end{equation*}
	induces a natural morphism
	\begin{equation*}
		\mathbb{L}\Delta^*(\M_1^\bullet\boxtimes \M_2^\bullet)\to \M_1^\bullet\widetilde{\otimes}_{\O_X}^{\mathbb{L}}\M_2^\bullet
	\end{equation*}
	by adjunction.
	
	If $\M_1^\bullet$, $\M_2^\bullet$ are bounded above, we can find suitable resolutions, where each term is of the form $\O_X\widetilde{\otimes}_K \F$ for some flat object in $\mathrm{Shv}(X, LH(\h{\B}c_K))$ (compare \cite[Theorem 3.24]{Bode6Op}). In this case, the morphism above reduces to the isomorphism
	\begin{equation*}
		(\O_X\widetilde{\otimes}_K \F_1)\widetilde{\otimes}_{\O_X}(\O_X\widetilde{\otimes}_K \F_2)\cong \O_X\widetilde{\otimes}_K (\F_1\widetilde{\otimes}_K \F_2).
	\end{equation*}
	In the general case, passing to homotopy colimits proves the result.
\end{proof}

Note that 
\begin{equation*}
\w{\D}_{X\times Y}\cong \w{\D}_X\boxtimes \w{\D}_Y
\end{equation*}
by arguing locally and picking local coordinates. As any $\w{\D}$-module admits a resolution by flat $\w{\D}$-modules which are also flat as $\O$-modules (\cite[subsection 7.1]{Bode6Op}), it follows that $\boxtimes$ can also be viewed as a functor $\mathrm{D}(\w{\D}_X)\times \mathrm{D}(\w{\D}_Y)\to \mathrm{D}(\w{\D}_{X\times Y})$.

This functor is particularly well-behaved on coadmissible modules:

\begin{lem}
	\label{Cbox}
The functor $\boxtimes$ restricts to an exact functor $\boxtimes: \C_X\times \C_Y\to \C_{X\times Y}$. Moreover, $\boxtimes$ preserves bounded above $\C$-complexes.
\end{lem} 
\begin{proof}
Suppose that $X$ and $Y$ are smooth affinoids admitting local coordinates $(x_1, \dots, x_r, \partial_1, \dots, \partial_r)$ and $(y_1, \dots, y_s, \xi_1, \dots, \xi_s)$, respectively. If $M=\varprojlim M_n$ is a coadmissible $\w{\D}_X(X)$-module and $N=\varprojlim N_n$ is a coadmissible $\w{\D}_Y(Y)$-module, then by \cite[Corollary 5.23]{Bode6Op}, 
\begin{equation*}
M\widetilde{\otimes}_K N\cong \varprojlim (M_n\widetilde{\otimes}_K N_n),
\end{equation*}
which is a coadmissible $\w{\D}_{X\times Y}(X\times Y)\cong \w{\D}_X(X)\widetilde{\otimes}_K \w{\D}_Y(Y)$-module.

If $U\subseteq X$, $V\subseteq Y$ are affinoid subdomains, then 
\begin{equation*}
(\w{\D}_X(U)\widetilde{\otimes}_K \w{\D}_Y(V))\widetilde{\otimes}_{\w{\D}(X)\widetilde{\otimes}\w{\D}_Y(Y)}(M\widetilde{\otimes}_K N)\cong (\w{\D}_X(U)\widetilde{\otimes}_{\w{\D}_X(X)}M)\widetilde{\otimes}_K (\w{\D}_Y(V)\widetilde{\otimes}_{\w{\D}_Y(Y)}N).
\end{equation*} 
Thus if $\M$ resp. $\N$ is a coadmissible $\w{\D}_X$- resp. $\w{\D}_Y$-module, associated to $M$ and $N$, respectively, then
\begin{equation*}
\w{\D}_{X\times Y}\widetilde{\otimes}_{p_1^{-1}\w{\D}_X\widetilde{\otimes}_K p_2^{-1}\w{\D}_Y}(p_1^{-1}\M\widetilde{\otimes}_K p_2^{-1}\N)
\end{equation*}
is the coadmissible $\w{\D}_{X\times Y}$-modules associated to $M\widetilde{\otimes}_K N$. In particular, this functor is exact by exactness of the completed tensor product (\cite[Proposition 4.21]{Bode6Op}).

Now by associativity of tensor products and by \cite[Proposition 5.33]{Bode6Op}, 
\begin{align*}
\M\boxtimes \N&\cong \w{\D}_{X\times Y}\widetilde{\otimes}^\mathbb{L}_{\w{\D}_X(X)\widetilde{\otimes}_K \w{\D}_Y(Y)} (\M(X)\widetilde{\otimes}_K \N(Y))\\
&\cong \w{\D}_{X\times Y} \widetilde{\otimes}_{\w{\D}_X(X)\widetilde{\otimes}_K \w{\D}_Y(Y)} (\M(X)\widetilde{\otimes}_K \N(Y))\\
&\cong \w{\D}_{X\times Y}\widetilde{\otimes}_{p_1^{-1}\w{\D}_X\widetilde{\otimes}_K p_2^{-1}\w{\D}_Y}(p_1^{-1}\M\widetilde{\otimes}_K p_2^{-1}\N)
\end{align*}
is a coadmissible $\w{\D}_{X\times Y}$-module concentrated in degree zero.

If $X=\Sp A$, $Y=\Sp B$, let $\A$, $\B$ denote affine formal models, and suppose without loss of generality that the $\A$-module spanned by the $\partial_i$ is an $(R, \A)$-Lie lattice in $\T_X(X)$, which we denote by $\L_1$. Similarly, we obtain $\L_2\subseteq \T_Y(Y)$. Now note that the image $\C$ of $\A\h{\otimes}_R\B$ inside $\O_{X\times Y}(X\times Y)=A\widetilde{\otimes}_K B$ yields an admissible affine formal model for $A\widetilde{\otimes}_K B$, and the $\C$-submodule of $\T_{X\times Y} (X\times Y)$ generated by the $\partial_i$ and the $\xi_i$ is an $(R, \C)$-Lie lattice. These choices of lattices provide us with sheaves $\D_{X, n}$, $\D_{Y, n}$, $\D_{X\times Y, n}$ such that
\begin{equation*}
\D_{X\times Y, n}\cong \D_{X, n}\boxtimes \D_{Y, n}.
\end{equation*}
The same argument as above then shows that $\boxtimes$ is an exact functor on coherent $\D_n$-modules, yielding the functor 
\begin{equation*}
\boxtimes: \mathrm{D}^b_{\mathrm{coh}}(\D_{X, n})\times  \mathrm{D}^b_{\mathrm{coh}}(\D_{Y, n})\to  \mathrm{D}^b_{\mathrm{coh}}(\D_{X\times Y, n}).
\end{equation*}
Moreover,
\begin{align*}
\D_{X, n}(X)\widetilde{\otimes}_K \D_{Y, n}(Y)\widetilde{\otimes}_{\w{\D}_X(X)\widetilde{\otimes}_K \w{\D}_Y(Y)} (\M(X)\widetilde{\otimes}_K \N(Y))\\
\cong (\D_{X, n}(X)\widetilde{\otimes}_{\w{\D}_X(X)}\M(X))\widetilde{\otimes}_K(\D_{Y, n}(Y)\widetilde{\otimes}_{\w{\D}_Y(Y)}\N(Y)) 
\end{align*}
implies
\begin{equation*}
\D_{X\times Y, n}\widetilde{\otimes}_{\w{\D}_{X\times Y}} (\M\boxtimes \N)\cong \M_n \boxtimes \N_n
\end{equation*}
for any coadmissible $\w{\D}_X$-module $\M$ and coadmissible $\w{\D}_Y$-module $\N$ with $\M_n=\D_{X, n}\widetilde{\otimes}_{\w{\D}_X} \M$, similarly for $\N$.

In particular, it follows from \cite[Lemma 6.6]{Bode6Op} that
\begin{equation*}
\M\boxtimes \N\cong \varprojlim \M_n\boxtimes \N_n \cong \varprojlim \D_{X\times Y, n}\widetilde{\otimes}_{\w{\D}_{X\times Y}}(\M\boxtimes \N).
\end{equation*}
Now if $\M^\bullet\in \mathrm{D}_{\C}^{-}(\w{\D}_X)$ and $\N^\bullet\in \mathrm{D}_{\C}^{-}(\w{\D}_Y)$, then by the above and \cite[tag 0132]{stacksproj}, $\mathrm{H}^i(\M^\bullet\boxtimes \N^\bullet)$ is given as the successive (finite) extension of the coadmissible modules
\begin{equation*}
	\mathrm{H}^j(\M^\bullet)\boxtimes \mathrm{H}^{i-j}(\N^\bullet), \ j\in \mathbb{Z}.
\end{equation*}
In particular, $\mathrm{H}^i(\M^\bullet\boxtimes \N^\bullet)$ is a coadmissible $\w{\D}_{X\times Y}$-module for each $i$.

Furthermore,
\begin{equation*}
	\mathrm{H}^i(\D_{X\times Y, n}\widetilde{\otimes}_{\w{\D}_{X\times Y}}^{\mathbb{L}}\M^\bullet\boxtimes \N^\bullet)
\end{equation*}
is given as the corresponding extension of
\begin{equation*}
	\mathrm{H}^j(\D_{X, n}\widetilde{\otimes}_{\w{\D}_X}^{\mathbb{L}}\M^\bullet)\boxtimes \mathrm{H}^{i-j}(\D_{Y, n}\widetilde{\otimes}_{\w{\D}_Y}^{\mathbb{L}}\N^\bullet),
\end{equation*}
so that $\D_{X\times Y, n}\widetilde{\otimes}_{\w{\D}_{X\times Y}}^{\mathbb{L}}(\M^\bullet\boxtimes \N^\bullet)$ is bounded for each $n$. Thus $\M^\bullet\boxtimes \N^\bullet$ is a bounded above $\C$-complex by \cite[Proposition 8.5]{Bode6Op}.
\end{proof}

The same argument as above shows that if $\M^\bullet$ is a $\C$-complex on $X$ and $\N^\bullet\in \mathrm{D}^b_\C(\w{\D}_Y)$, then $\M^\bullet\boxtimes \N^\bullet\in \mathrm{D}_\C(\w{\D}_{X\times Y})$. It is however evident that $\boxtimes$ does not preserve arbitrary (unbounded) $\C$-complexes. Analogously to \cite[Remark after Theorem 6.9]{Bodegldim}, one might improve the situation in an $\infty$-categorical framework by glueing the exterior products for coherent $\D_n$-modules, but we will not investigate this further.

\begin{lem}
Let $X$ be a smooth rigid analytic $K$-variety. Let $\Delta: X\to X\times X$ denote the diagonal map. There is a natural isomorphism
\begin{equation*}
\M_1^\bullet\widetilde{\otimes}^\mathbb{L}_{\O_X} \M_2^\bullet\cong \Delta^!(\M_1^\bullet\boxtimes \M_2^\bullet)[\mathrm{dim}X]
\end{equation*}
for any $\M_1^\bullet, \M_2^\bullet\in \mathrm{D}(\w{\D}_X)$.
\end{lem}
\begin{proof}
We have a natural morphism 
\begin{equation*}
\Delta^!(\M_1^\bullet\boxtimes \M_2^\bullet)[\mathrm{dim}X]\to \M_1^\bullet\widetilde{\otimes}^\mathbb{L}_{\O_X} \M_2^\bullet
\end{equation*}
induced by the map
\begin{equation*}
\M_1^\bullet\boxtimes \M_2^\bullet\to \mathrm{R}\Delta_*(\M_1^\bullet\widetilde{\otimes}^\mathbb{L}_{\O_X} \M_2^\bullet) 
\end{equation*}
via adjunction.

This is an isomorphism by Lemma \ref{boxtimesOdiagonal}.
\end{proof}

\begin{lem}
	\label{tensorofpullbacks}
There is a natural isomorphism
\begin{equation*}
\M_1^\bullet\boxtimes \M_2^\bullet\cong p_1^!\M_1^\bullet\widetilde{\otimes}^\mathbb{L}_{\O_{X_1\times X_2}}p_2^!\M_2^\bullet[-\mathrm{dim}X_1-\mathrm{dim}X_2].
\end{equation*}
for $\M_1^\bullet\in \mathrm{D}(\w{\D}_{X_1})$, $\M_2^\bullet\in \mathrm{D}(\w{\D}_{X_2})$.
\end{lem}
\begin{proof}
This is immediate from the definition.
\end{proof}

\begin{cor}
	\label{boxwithO}
	There is a natural isomorphism
	\begin{equation*}
		\O_{X_1}\boxtimes \M^\bullet\cong p_2^!\M^\bullet[-\mathrm{dim}X_2]
	\end{equation*}
	for $\M^\bullet\in \mathrm{D}(\w{\D}_{X_2})$.
\end{cor}

\begin{prop}
\label{boxdirect}
Let $f: X_1\to Y$ be a projective morphism of smooth rigid analytic $K$-varieties, and let $X_2$ be another smooth rigid analytic $K$-variety. If $\M^\bullet\in \mathrm{D}_\C^-(\w{\D}_{X_1})$, and $\N^\bullet\in \mathrm{D}_\C^-(\w{\D}_{X_2})$, then there is a natural isomorphism
\begin{equation*}
 (f_+\M^\bullet)\boxtimes \N^\bullet\cong (f\times \mathrm{id})_+(\M^\bullet\boxtimes \N^\bullet).
\end{equation*}
The same holds for $\M^\bullet\in \mathrm{D}_{\C}(\w{\D}_{X_1})$ if $\N^\bullet\in \mathrm{D}_{\C}^b(\w{\D}_{X_2})$.
\end{prop}
\begin{proof}
The universal property of tensor products yields a natural morphism
\begin{equation*}
	(f_+\M^\bullet)\boxtimes \N^\bullet\to (f\times \mathrm{id})_+(\M^\bullet\boxtimes \N^\bullet).
\end{equation*}	
 
Note that since closed embeddings preserve $\C$-complexes, we can use the composition results in Proposition \ref{compandadj} to restrict to the cases where $f$ is either a closed embedding or a projection $\mathbb{P}^{d, \mathrm{an}}_K\times Y\to Y$. 

The case of a closed embedding is straightforward. As all functors involved are exact on coadmissible modules in this case, we can reduce to the case where $\M^\bullet$ and $\N^\bullet$ are concentrated in degree $0$. Locally, we can assume that $Y$ admits a local coordinate system $(x_1, \dots, x_m, \partial_1, \dots, \partial_m)$ such that $X_1$ is given by the vanishing of $x_{s+1}, \dots, x_m$ for some $s<r$. Assuming also that $X_2$ is affinoid, $\M^\bullet$ can be represented by a coadmissible $\w{\D}_{X_1}$-module, which is the localisation of a coadmissible $\w{\D}_{X_1}(X_1)$-module $M$, and likewise $\N^\bullet$ is represented by a coadmissible $\w{\D}_{X_2}$-module, the localisation of a coadmissible $\w{\D}_{X_2}(X_2)$-module $N$. The natural morphism on global sections reads
\begin{align*}
\left(\frac{\w{\D}_{Y}(Y)}{\w{\D}_Y(Y)\cdot(x_{s+1}, \dots, x_m)}\widetilde{\otimes}_{\w{\D}_{X_1}(X_1)}M\right)\widetilde{\otimes}_K N\\\cong K\{\partial_{s+1}, \dots, \partial_m\}\widetilde{\otimes}_K M\widetilde{\otimes}_K N\\
\cong\frac{\w{\D}_{Y\times X_2}(Y\times X_2)}{\w{\D}_{Y\times X_2}(Y\times X_2)(x_{s+1}, \dots, x_m)}\widetilde{\otimes}_{\w{\D}_{X_1\times X_2}(X_1\times X_2)}(M\widetilde{\otimes}_K N), 
\end{align*}
as required.

Now consider the case where $f: X_1=\mathbb{P}^{d, \mathrm{an}}_K\times Y\to Y$ is the natural projection. Arguing locally, we can assume that $Y$ and $X_2$ are D-regular affinoids.

As $\w{\D}_{Y\leftarrow X_1}\widetilde{\otimes}_{\w{\D}_{X_1}}-$ has finite cohomological dimension on coadmissible modules, likewise $\w{\D}_{Y\times X_2\leftarrow X_1\times X_2}\widetilde{\otimes}_{\w{\D}_{X_1\times X_2}}-$ by the side-changed version of \cite[Lemma A.9]{Bodegldim}, we can further reduce to the case where $\M^\bullet=\M$ is a coadmissible $\w{\D}_{X_1}$-module and $\N^\bullet=\N$ is a coadmissible $\w{\D}_{X_2}$-module.

Now
\begin{equation*}
	\mathrm{H}^j(f_+\M\boxtimes \N)\cong \mathrm{H}^j(f_+\M)\boxtimes \N
\end{equation*}
is the coadmissible $\w{\D}_{Y\times X_2}$-module associated to the global sections
\begin{equation*}
	\mathrm{H}^j(f_+\M)(Y)\widetilde{\otimes}_K \N(X_2).
\end{equation*}

By the same argument as in \cite[subsection A.2]{Bodegldim}, $\mathrm{H}^j(f_+\M)(Y)$ can be calculated as follows: the transfer bimodule $\w{\D}_{Y\leftarrow X_1}$ admits a Koszul--Spencer resolution
\begin{equation*}
	\mathcal{S}^\bullet=(\Omega^\bullet_{\mathbb{P}^{d, \mathrm{an}}_K}\widetilde{\otimes}_{\O_{\mathbb{P}^{d, \mathrm{an}}_K}}\w{\D}_{\mathbb{P}^{d, \mathrm{an}}_K})\boxtimes \w{\D}_Y
\end{equation*}
so that
\begin{equation*}
	\mathrm{H}^j(f_+(\M))(Y)=\mathrm{H}^j(\mathrm{Tot}(\check{C}^{\bullet}(\mathfrak{U}\times Y, \mathcal{S}^\bullet\widetilde{\otimes}_{\w{\D}_{\mathbb{P}^{d, \mathrm{an}}_K\times Y}}\M))
\end{equation*}
for any finite affinoid covering $\mathfrak{U}$ of $\mathbb{P}^{d, \mathrm{an}}$.

But now in the same way, sections of $\mathrm{H}^j((f\times \mathrm{id})_+(\M\boxtimes \N))$ are calculated by the corresponding Cech complex relative to a covering $\mathfrak{U}\times Y\times X_2$, which is obtained from the one above by tensoring with $\N(X_2)$ over $K$. We thus obtain the desired isomorphism
\begin{equation*}
	\mathrm{H}^j((f\times\mathrm{id})_+(\M\boxtimes \N))(Y\times X_2)\cong \mathrm{H}^j(f_+\M)(Y)\widetilde{\otimes}_K \N(X_2),
\end{equation*}
as required.
\end{proof}

We remark that the same argument applies much more generally for managable $f$ if we restrict to $\C$-complexes $\M^\bullet$ such that $f_+\M^\bullet\in \mathrm{D}_\C(\w{\D}_Y)$ and $f_+\mathrm{H}^j(\M^\bullet)\in \mathrm{D}_\C(\w{\D}_Y)$ for each $j$ -- the latter is necessary so that we can reduce to the case of a single module as above.

\begin{lem}
	\label{pullbackmonoidal}
Let $f: X\to Y$ be a morphism of smooth rigid analytic $K$-varieties. Then there is a natural isomorphism
\begin{equation*}
f^!(\M^\bullet\widetilde{\otimes}^\mathbb{L}_{\O_Y} \N^\bullet)\cong f^!\M^\bullet\widetilde{\otimes}^\mathbb{L}_{\O_X} f^!\N^\bullet[\mathrm{dim}Y-\mathrm{dim}X]
\end{equation*}
for any $\M^\bullet, \N^\bullet\in \mathrm{D}(\w{\D}_Y)$.
\end{lem}
\begin{proof}
The natural morphism is an isomorphism, as
\begin{equation*}
\O_X\widetilde{\otimes}^\mathbb{L}_{f^{-1}\O_Y} f^{-1}(\M^\bullet\widetilde{\otimes}^\mathbb{L}_{\O_Y} \N^\bullet)\cong (\O_X\widetilde{\otimes}^\mathbb{L}_{f^{-1}\O_Y} f^{-1}\M^\bullet)\widetilde{\otimes}^\mathbb{L}_{\O_X}(\O_X\widetilde{\otimes}^\mathbb{L}_{f^{-1}\O_Y} f^{-1} \N^\bullet).
\end{equation*}
\end{proof}

\begin{prop}
Let $f:X_1\to Y_1$, $g: X_2\to Y_2$ be morphisms of smooth rigid analytic $K$-varieties.  If $\M^\bullet\in \mathrm{D}(\w{\D}_{Y_1})$, $\N^\bullet\in \mathrm{D}(\w{\D}_{Y_2})$, then there is a natural isomorphism
\begin{equation*}
(f\times g)^!(\M^\bullet\boxtimes \N^\bullet)\cong (f^!\M^\bullet)\boxtimes (g^!\N^\bullet).
\end{equation*}
\end{prop}
\begin{proof}
Let $p_i: X_1\times X_2\to X_i$, $q_i: Y_1\times Y_2\to Y_i$ denote the natural projections for $i=1, 2$.

By Lemma \ref{tensorofpullbacks}, 
\begin{equation*}
f^!\M^\bullet\boxtimes g^!\N^\bullet\cong p_1^!f^!\M^\bullet\widetilde{\otimes}^\mathbb{L}_{\O_{X_1\times X_2}}p_2^!g^!\N^\bullet[-\mathrm{dim}X_1-\mathrm{dim}X_2].
\end{equation*}
Since 
\begin{equation*}
p_1^!f^!\cong (fp_1)^!\cong (f\times g)^!q_1^!
\end{equation*} 
by \cite[Proposition 7.5]{Bode6Op}, and likewise
\begin{equation*}
p_2^!g^!\cong (gp_2)^!\cong (f\times g)^!q_2^!,
\end{equation*}
this yields
\begin{align*}
f^!\M^\bullet\boxtimes g^!\N^\bullet&\cong (f\times g)^!q_1^!\M^\bullet\widetilde{\otimes}^\mathbb{L}_{\O_{X_1\times X_2}} (f\times g)^!q_2^!\N^\bullet[-\mathrm{dim}X_1-\mathrm{dim}X_2]\\
&\cong (f\times g)^!(q_1^!\M^\bullet\widetilde{\otimes}^\mathbb{L}_{\O_{Y_1\times Y_2}}q_2^!\N^\bullet)[-\mathrm{dim}Y_1-\mathrm{dim}Y_2],
\end{align*}
where the last isomorphism follows from Lemma \ref{pullbackmonoidal}. The result follows by applying Lemma \ref{tensorofpullbacks} for a second time.
\end{proof}

\subsection{Further properties of the duality functor}

Recall from \cite[Theorem 9.11]{Bode6Op} that for any smooth morphism $f$ of smooth rigid analytic $K$-varieties, $f^!$ is exact on coadmissible modules and preserves $\C$-complexes. This subsection is devoted to the following theorem:

\begin{thm}
	\label{pullbackdual}
	Let $f:X\to Y$ be a smooth morphism of smooth rigid analytic $K$-varieties. Then there is a natural isomorphism
	\begin{equation*}
		\mathbb{D}f^!\M^\bullet\cong f^!\mathbb{D}\M^\bullet
	\end{equation*}
	for any $\M^\bullet\in \mathrm{D}^b_\C(\w{\D}_Y)$. In particular, $f^!\M^\bullet\cong f^+\M^\bullet$ in that case.
\end{thm}

\begin{proof}
Let $\M^\bullet\in\mathrm{D}^b_\C(\w{\D}_Y)$.

By \cite[Proposition 6.25]{Bodegldim}, there is a natural isomorphism 

\begin{equation*}
\mathrm{R}\mathcal{H}om_{\w{\D}_Y}(\M^\bullet, \M^\bullet)\cong \Omega_Y\widetilde{\otimes}^\mathbb{L}_{\w{\D}_Y}(\mathbb{D}\M^\bullet\widetilde{\otimes}^\mathbb{L}_{\O_Y} \M^\bullet)[-\mathrm{dim}Y].
\end{equation*}
Applying \cite[Proposition 6.22]{Bodegldim} yields a natural isomorphism
\begin{equation*}
\mathrm{R}\mathcal{H}om_{\w{\D}_Y}(\M^\bullet, \M^\bullet)\cong \mathrm{R}\mathcal{H}om_{\w{\D}_Y}(\O_Y, \mathbb{D}\M^\bullet\widetilde{\otimes}^\mathbb{L}_{\O_Y} \M^\bullet).
\end{equation*}
We thus obtain the natural isomorphism
\begin{equation*}
\mathrm{Hom}_{\mathrm{D}(\w{\D}_Y)}(\M^\bullet, \M^\bullet)\cong \mathrm{Hom}_{\mathrm{D}(\w{\D}_Y)}(\O_Y, \mathbb{D}\M^\bullet\widetilde{\otimes}^\mathbb{L}_{\O_Y} \M^\bullet)
\end{equation*}
by taking $\mathrm{H}^0\circ \mathrm{R}\Gamma$ of the above.

This yields the natural morphism
\begin{align*}
\mathrm{Hom}_{\mathrm{D}(\w{\D}_Y)}(\M^\bullet, \M^\bullet)&\cong \mathrm{Hom}_{\mathrm{D}(\w{\D}_Y)}(\O_Y, \mathbb{D}\M^\bullet\widetilde{\otimes}^\mathbb{L}_{\O_Y} \M^\bullet)\\
&\to \mathrm{Hom}_{\mathrm{D}(\w{\D}_X)}(\mathbb{L}f^*\O_Y, \mathbb{L}f^*\mathbb{D}\M^\bullet\widetilde{\otimes}^\mathbb{L}_{\O_X} \mathbb{L}f^*\M^\bullet)\\
&\cong \mathrm{Hom}_{\mathrm{D}(\w{\D}_X)}(\O_X, \mathbb{L}f^*\M^\bullet\widetilde{\otimes}_{\O_X}^{\mathbb{L}}\mathbb{L}f^*\mathbb{D}\M^\bullet)\\
&\cong \mathrm{Hom}_{\mathrm{D}(\w{\D}_X)}(\O_X, \mathbb{D}^2\mathbb{L}f^*\M^\bullet\widetilde{\otimes}^\mathbb{L}_{\O_X}\mathbb{L}f^*\mathbb{D}\M^\bullet)\\
&\cong \mathrm{Hom}_{\mathrm{D}(\w{\D}_X)}(\mathbb{D}\mathbb{L}f^*\M^\bullet, \mathbb{L}f^*\mathbb{D}\M^\bullet),
\end{align*}
so that the identity morphism $\M^\bullet\to \M^\bullet$ gives rise to a natural morphism $\mathbb{D}\mathbb{L}f^*\M^\bullet\to \mathbb{L}f^*\mathbb{D}\M^\bullet$. Applying the shift $[\mathrm{dim}X-\mathrm{dim}Y]$, this gives a natural morphism
\begin{equation*}
\mathbb{D}f^!\M^\bullet\to f^!\mathbb{D}\M^\bullet.
\end{equation*}
To prove that this is an isomorphism, we can assume that $X=\Sp B$ and $Y=\Sp A$ are smooth affinoids with $X$ admitting a coordinate system $(x_1, \dots, x_m, \partial_1, \dots, \partial_m)$ with the property that
\begin{equation*}
f^*\T_Y\cong \T_X/\oplus_{i=d+1}^m \O_X\partial_i,
\end{equation*}
where $d=\mathrm{dim}Y$. In particular,  the $B$-module $B\otimes_A \T_Y(Y)$ is a direct summand of $\T_X(X)$. By \cite[Theorem 1.1]{Bodegldim}, we can assume that $A$ admits an affine formal model $\A$ and an $(R, \A)$-Lie lattice $\L'$ such that $\h{U_\A(\pi^n\L')}_K$ is Auslander regular for all $n\geq 0$.  Choose an affine formal model $\B$ of $B$ containing the image of $\A$, then there exists a finitely generated $\B$-submodule $\L\subseteq \T_X(X)$ such that $K\otimes_R \L=\T_X(X)$ and the image of $\L$ in $B\otimes_A \T_Y(Y)$ is precisely $\B\otimes_\A \L'$. Rescaling $\L'$ and $\L$, we can assume that $\L$ is an $(R, \B)$-Lie lattice. We denote the resulting sheaves by $\D_{X_n}$, $\D_{Y_n}$, respectively.

We wish to show that the morphism 
\begin{equation*}
\mathbb{D}f^!\M^\bullet\to f^!\mathbb{D}\M^\bullet
\end{equation*}  
of $\C$-complexes is an isomorphism, so it suffices to show that
\begin{equation*}
\D_{X_n}\widetilde{\otimes}^\mathbb{L}_{\w{\D}_X} \mathbb{D}f^!\M^\bullet \to \D_{X_n}\widetilde{\otimes}^\mathbb{L}_{\w{\D}_X} f^!\mathbb{D}\M^\bullet
\end{equation*}
is an isomorphism for all $n$, by \cite[Corollary 8.17]{Bode6Op}.

By the proof of \cite[Theorem 9.17]{Bode6Op}, we have
\begin{equation*}
\D_{X_n}\widetilde{\otimes}^\mathbb{L}_{\w{\D}_X} \mathbb{D}\N^\bullet \cong \mathrm{R}\mathcal{H}om_{\D_{X_n}}(\D_{X_n}\widetilde{\otimes}^\mathbb{L}_{\w{\D}_X}\N^\bullet, \D_{X_n})\widetilde{\otimes}_{\O_X} \Omega_X^{-1}[\mathrm{dim}X]
\end{equation*}
for any $\C$-complex $\N^\bullet$ on $X$. We denote the right hand side by $\mathbb{D}_n(\N_n^\bullet)$ to simplify the notation.

Furthermore, the isomorphism
\begin{align*}
f^*\D_{Y_n}&\cong \D_{X_n}/\sum_{i=d+1}^m\D_{X_n}\cdot \partial_i\\
&\cong \D_{X_n}\widetilde{\otimes}^\mathbb{L}_{\w{\D}_X} f^*\w{\D}_X
\end{align*}
shows that
\begin{equation*}
\D_{X_n}\widetilde{\otimes}^\mathbb{L}_{\w{\D}_X}f^*\M^\bullet\cong f^*(\D_{Y_n}\widetilde{\otimes}^\mathbb{L}_{\w{\D}_Y}\M^\bullet)
\end{equation*}
naturally.

Thus, the morphism above can be written as 
\begin{equation*}
\mathbb{D}_n\mathbb{L}f^*(\M_n^\bullet)[\mathrm{dim}X-\mathrm{dim}Y]\to \mathbb{L}f^*\mathbb{D}_n(\M_n^\bullet)[\mathrm{dim}X-\mathrm{dim}Y] 
\end{equation*}
for $\M_n^\bullet:=\D_{Y_n}\widetilde{\otimes}^\mathbb{L}_{\w{\D}_Y} \M^\bullet\in \mathrm{D}^b_\mathrm{coh}(\D_{Y_n})$.

By Auslander regularity of $\D_{Y_n}(Y)$, it now suffices to show the isomorphism in the case when $\M_n^\bullet=\D_{Y_n}$. Tensoring both sides with $\Omega_X\widetilde{\otimes}_{\O_X}$, the left hand side becomes
\begin{equation*}
\mathrm{R}\mathcal{H}om_{\D_{X_n}}(\D_{X_n}/\sum_{i=d+1}^m \D_{X_n}\cdot \partial_i, \D_{X_n})[2\mathrm{dim}X-\mathrm{dim}Y],
\end{equation*}
which can be calculated by a Spencer--Koszul resolution, while the right hand side is the right $\D_{X_n}$-module
\begin{equation*}
\Omega_X\widetilde{\otimes}_{\O_X}f^*\D_{Y_n}\widetilde{\otimes}_{\O_X}\Omega_X^{-1}[\mathrm{dim}X]\cong f^*(\D_{Y_n}^\mathrm{op})[\mathrm{dim}X].
\end{equation*} 
The desired isomorphism thus reduces to the statement that the dual of the Koszul resolution provides a resolution of the right module $\D_{X_n}/\sum_{i=d+1}^m \partial_i\cdot \D_{X_n}$, shifted by $\mathrm{dim}X-\mathrm{dim}Y=m-d$. 
\end{proof}

We remark that the boundedness assumption on $\M^\bullet$ was only needed to produce the natural morphism, using the isomorphisms from \cite{Bodegldim} -- it is quite likely that this holds more generally for arbitrary $\C$-complexes.

As far as direct images are concerned, the direct image functor $f_+$ for $f$ projective tends to exhibit similar properties to $f^!$ for $f$ smooth. Accordingly, we show the following in \cite{Bodepoincare}:
\begin{prop}
	\label{Poincare}
	Let $f: X\to Y$ be a projective morphism of smooth rigid analytic $K$-varieties. There is a natural isomorphism
	\begin{equation*}
		f_+\mathbb{D}\M^\bullet\to \mathbb{D}f_+\M^\bullet
	\end{equation*}
	for $\M^\bullet\in \mathrm{D}_{\C}(\w{\D}_X)$. In particular, $f_+\M^\bullet\cong f_!\M^\bullet$ for all $\M^\bullet\in \mathrm{D}_{\C}(\w{\D}_X)$.
\end{prop}

\section{$0$-holonomic $\C$-complexes}
\subsection{Definition and first properties}
Let $X$ be a smooth rigid analytic $K$-variety.
\begin{defn}
A $\C$-complex $\M^\bullet$ on $X$ is called \textbf{$0$-holonomic} if for any smooth morphism $f: X'\to X$ and any divisor $Z\subseteq X'$, we have $\mathrm{R}\underline{\Gamma}_Z(f^!\M^\bullet)\in \mathrm{D}_\C(\w{\D}_{X'})$.
\end{defn}
We denote by $\mathrm{D}_{0-\mathrm{hol}}(\w{\D}_X)$ the full subcategory of $\mathrm{D}(\w{\D}_X)$ consisting of all $0$-holonomic $\C$-complexes on $X$.
\begin{lem}
	\label{zeroholtriang}
	The subcategory $\mathrm{D}_{0-\mathrm{hol}}(\w{\D}_X)$ is a triangulated subcategory of $\mathrm{D}(\w{\D}_X)$.
\end{lem}
\begin{proof}
	This follows directly from the fact that $\mathrm{D}_\C(\w{\D}_X)$ is triangulated (by \cite[Proposition 8.4]{Bode6Op}) and the functors $f^!$, $\mathrm{R}\underline{\Gamma}_Z$ preserve distinguished triangles.
\end{proof}

Recall that a vector bundle with integrable connection (or for short, an integrable connection) is an $\O$-coherent $\w{\D}$-module.

\begin{prop}
\label{ICzero}
Integrable connections are $0$-holonomic.
\end{prop}
\begin{proof}
If $\M$ is an integrable connection on $X$ and $f: X'\to X$ is smooth, then $f^!\M$ is an integrable connection on $X'$, shifted in degree. Hence the result follows from \cite[Theorem 10.6]{DcapThree}.
\end{proof}

\begin{lem}
\label{anyclosedsubvar}
Let $\M^\bullet$ be a $0$-holonomic $\C$-complex. Then for any smooth morphism $f: X'\to X$ and any Zariski closed subvariety $Z\subseteq X'$, we have $\mathrm{R}\underline{\Gamma}_Z(f^!\M^\bullet)\in \mathrm{D}_\C(\w{\D}_{X'})$.
\end{lem}
\begin{proof}
Writing $Z$ as an intersection of divisors, this follows from Proposition \ref{MVforDcap} and Lemma \ref{zeroholtriang}.
\end{proof}
In particular, if $\M^\bullet$ is $0$-holonomic on $X$ and $f: X'\to X$ is smooth, $Z\subseteq X'$ a Zariski closed subvariety with complement $j: U\to X'$, then $j_+j^!\M^\bullet=j_+(\M^\bullet|_U)$ is a $\C$-complex by Proposition \ref{triangle}.

It is immediate that $0$-holonomicity is stable under $f^!$ whenever $f$ is smooth: $f^!$ preserves $\C$-complexes in this case by Theorem \ref{Ccomplexprops}.(ii), and the result follows directly from \cite[Proposition 7.5]{Bode6Op} and the fact that the composition of two smooth morphisms is smooth.
\begin{lem}
	Let $X$ be a smooth rigid analytic $K$-variety and let $\{U_i\}_{i\in I}$ be an admissible covering. Then $\M^\bullet\in \mathrm{D}(\w{\D}_X)$ is $0$-holonomic if and only if $\M^\bullet|_{U_i}$ is $0$-holonomic for all $i$.
\end{lem}
\begin{proof}
	One direction is proved by the previous remark.
	
	Conversely, if $\M^\bullet|_{U_i}$ is a $0$-holonomic $\C$-complex for all $i$, then $\M^\bullet$ is a $\C$-complex by definition. If $f: X'\to X$ is a smooth morphism, it suffices to work on the admissible covering $\{U_i\times_X X'\}$ to check that a given complex is a $\C$-complex, and the result follows, as $U_i\times_X X'\to U_i$ is smooth and $Z\cap (U_i\times_X X')$ is a divisor or empty.
\end{proof}

\begin{cor}
	\label{truncatezerohol}
	Let $X$ be a smooth rigid analytic $K$-variety. If $\M^\bullet$ is a $0$-holonomic $\C$-complex on $X$, then so is $\tau^{\geq n}\M^\bullet$, $\tau^{\leq n} \M^\bullet$ for all $n$. 
	
	In particular, $\mathrm{H}^q(\M^\bullet)$ is a $0$-holonomic $\w{\D}_X$-module for all $q$.
\end{cor}
\begin{proof}
	 If $f: X'\to X$ is smooth, and $j: U\to X'$ is the complement of a divisor, it follows from \cite[Lemma 3.5]{Bode6Op} that
	\begin{equation*}
		 j_+j^!f^!(\tau^{\geq n}\M^\bullet)\cong \tau^{\geq n} j_+j^!f^!\M^\bullet,
	\end{equation*}
	similarly for $\tau^{\leq n}$, since $f^!$, $j^!$ and $j_*$ are exact on coadmissible modules in this case, thanks to \cite[Proposition 10.1]{DcapThree}, \cite[Proposition 9.10]{Bode6Op}. Thus the result follows from \cite[Corollary 8.6]{Bode6Op}.
\end{proof}

\begin{cor}
	\label{zeroholcohom}
	A bounded $\C$-complex $\M^\bullet\in \mathrm{D}^b_{\C}(\w{\D}_X)$ is $0$-holonomic if and only if $\mathrm{H}^q(\M^\bullet)$ is $0$-holonomic for all $q$.
\end{cor}
\begin{proof}
	One direction is given by the Corollary above. If $\M^\bullet$ is a $\C$-complex, $f:X'\to X$ is smooth and $j: U\to X'$ is the complement of a divisor, the same argument as above shows that
	\begin{equation*}
		\mathrm{H}^q(j_+j^!f^!\M^\bullet)\cong j_+j^!f^!(\mathrm{H}^{q+\mathrm{dim}X'-\mathrm{dim}X}(\M^\bullet))
	\end{equation*}
	for any $q$. In particular, if $\mathrm{H}^q(\M^\bullet)$ is $0$-holonomic for each $q$, then $j_+j^!f^!\M^\bullet$ is a bounded complex with coadmissible cohomology, and $\M^\bullet$ is $0$-holonomic by \cite[Corollary 8.7]{Bode6Op}. 
\end{proof}

See Lemma \ref{unboundedconnection} for a counterexample in the unbounded case.

We will now establish further stability properties of $0$-holonomic complexes.

\begin{prop}
\label{pullbackalongclosed}
Let $i: Z\to X$ be a Zariski closed smooth subvariety. If $\M^\bullet$ is $0$-holonomic on $X$, then $i^!\M^\bullet$ is $0$-holonomic on $Z$.
\end{prop}
\begin{proof}
We first note that $i^!\M^\bullet$ is a $\C$-complex by the following reasoning:

Note that it follows from Lemma \ref{anyclosedsubvar} that $\mathrm{R}\underline{\Gamma}_Z(\M^\bullet)$ is a $\C$-complex. Therefore $j_+j^!\M^\bullet$ is a $\C$-complex by Proposition \ref{triangle}, where $j$ denotes the embedding of the complement of $Z$ in $X$, and thus $i_+i^!\M^\bullet\cong \mathrm{R}\underline{\Gamma}_Z(\M^\bullet)$ by Proposition \ref{triangle}.(iii). Thus $i^!\M^\bullet$ is a $\C$-complex by Proposition \ref{Kashiwaratest}.

Now let $f: Z'\to Z$ be smooth, $D\subseteq Z'$ a divisor with complement $j: U\to Z'$. By Proposition \ref{triangle}, it suffices to show that $j_+j^!f^!i^!\M^\bullet$ is a $\C$-complex.

Arguing locally on $Z$, we can take a tubular neighbourhood of $Z$ in $X$ (\cite[Theorem 1.18]{Kiehl}) so that without loss of generality $X\cong Z\times \mathbf{B}^d$, where $\mathbf{B}^d=\Sp K\langle x_1, \dots, x_d\rangle$ denotes the closed unit disc of dimension $d$. We thus have the commutative diagram
\begin{equation*}
\begin{xy}
\xymatrix{
U\ar[r]^s \ar[d]_j& U\times \mathbf{B}^d\ar[d]^{j'=j\times \mathrm{id}}& Y\ar[l]_{\phi}\\
Z'\ar[d]_f\ar[r]^{i'} & Z'\times \mathbf{B}^d\ar[d]^{f'=f\times \mathrm{id}}\\
Z\ar[r]^i & Z\times \mathbf{B}^d,
}
\end{xy}
\end{equation*} 
where $Y$ is the complement of $U\cong U\times \{0\}$ in $U\times \mathbf{B}^d$.\\
By Proposition \ref{triangle} and the fact that $\M^\bullet$ is $0$-holonomic, we have a distinguished triangle of $\C$-complexes
\begin{equation*}
s_+s^!j'^!f'^!\M^\bullet\to j'^!f'^!\M^\bullet\to \phi_+\phi^!j'^!f'^!\M^\bullet\to,
\end{equation*}
and hence a distinguished triangle
\begin{equation*}
j'_+s_+s^!j'^!f'^!\M^\bullet\to j'_+j'^!f'^!\M^\bullet\to j'_+\phi_+\phi^!j'^!f'^!\M^\bullet\to.
\end{equation*}
But both the second and the third term are $\C$-complexes on $Z'\times\mathbf{B}^d$ by the remark after Lemma \ref{anyclosedsubvar}, since $\M^\bullet$ is $0$-holonomic. Thus
\begin{equation*}
i'_+j_+j^!f^!i^!\M^\bullet\cong j'_+s_+s^!j'^!f'^!\M^\bullet
\end{equation*}
is a $\C$-complex, where the isomorphism above is a consequence of the composition results in Proposition \ref{compandadj}. It follows by Proposition \ref{Kashiwaratest} that $j_+j^!f^!i^!\M^\bullet$ is a $\C$-complex, as required.
\end{proof}

\begin{cor}
	\label{zeroholinv}
	Let $f: X\to Y$ be a morphism between smooth rigid analytic varieties. Then $f^!$ preserves $0$-holonomicity.
\end{cor}
\begin{proof}
	We can write $f$ as the composition of a closed embedding and a smooth morphism between smooth varieties. By the composition law for inverse images (\cite[Proposition 7.5]{Bode6Op}), the result follows from the definition of $0$-holonomicity (in the smooth case) and Proposition \ref{pullbackalongclosed} (in the closed case).
\end{proof}

\subsection{Base change}

\begin{thm}
\label{basechange}
Consider the Cartesian diagram
\begin{equation*}
\begin{xy}
\xymatrix{
Y\times_X Z\ar[r]^{\widetilde{g}} \ar[d]_{\widetilde{f}}& Y\ar[d]^f\\
Z \ar[r]^g& X
}
\end{xy}
\end{equation*} 
where we assume that all spaces involved are smooth rigid analytic $K$-varieties, and that $f$ is projective.

Let $\M^\bullet\in \mathrm{D}_\C(\w{\D}_Y)$ be a $0$-holonomic $\C$-complex with the property that $g^!f_+\M^\bullet$ is a $\C$-complex.

Then
\begin{equation*}
g^!f_+\M^\bullet\cong \widetilde{f}_+\widetilde{g}^!\M^\bullet.
\end{equation*}
\end{thm}
\begin{proof}
Consider the diagram
\begin{equation*}
\begin{xy}
\xymatrix{
Y\times_X Z \ar[r] \ar[d] & Y\times Z\ar[r]^{\widetilde{p}} \ar[d] & Y\ar[d]^f\\
Z\ar[r]^i & X\times Z\ar[r]^p & X
}
\end{xy}
\end{equation*}
where $\widetilde{p}: Y\times Z\to Y$ and $p: X\times Z\to X$ denote the natural projections, and $i: Z\to X\times Z$ is the graph embedding of $g$.

By smoothness of the projection maps, Theorem \ref{Ccomplexprops} implies that $\widetilde{p}^!\M^\bullet$ and $p^!f_+\M^\bullet$ are $\C$-complexes. It then follows from Corollary \ref{boxwithO} and Proposition \ref{boxdirect} that
\begin{align*}
(f\times \mathrm{id}_Z)_+\widetilde{p}^!\M^\bullet&=(f\times \mathrm{id}_Z)_+(\M^\bullet\boxtimes \O_Z)[\mathrm{dim} Z]\\
& \cong f_+\M^\bullet\boxtimes \O_Z[\mathrm{dim} Z]\\
& \cong p^!f_+\M^\bullet.
\end{align*}
Now note that $\widetilde{p}^!\M^\bullet$ is a $0$-holonomic $\C$-complex satisfying the same assumptions as $\M^\bullet$ itself, so it now suffices to consider the graph morphism $i: Z\to X\times Z$, due to the composition result Proposition \ref{compandadj}. 

Let $U\subseteq X\times Z$ be an admissible open such that the graph embedding factors through a closed embedding $Z\to U$. Consider the Cartesian diagram
\begin{equation*}
	\begin{xy}
		\xymatrix{
		Y\times_XZ\ar[r]\ar[d]& (Y\times Z)\times_{X\times Z}U\ar[r]^{\widetilde{j}}\ar[d]^{f'}& Y\times Z\ar[d]^{f\times \mathrm{id}}\\
		Z\ar[r]^{i'}&U\ar[r]^j &X\times Z,}
	\end{xy}
\end{equation*}
where $i=ji'$.

Now
\begin{align*}
	j^!(f\times \mathrm{id})_+\M^\bullet&\cong ((f\times \mathrm{id})_+\M^\bullet)|_U\\
	&\cong f'_+(\M^\bullet|_{(f\times \mathrm{id})^{-1}U})\\
	&\cong f'_+j'^!\M^\bullet.
\end{align*}
In particular, $j'^!\M^\bullet$ is a $0$-holonomic $\C$-complex on $Y\times_X U$ satisfying the assumptions of the Theorem, so we have reduced to the case where $g=i: Z\to X$ is a closed embedding.

Assume from now on that $g=i: Z\to X$ is a closed embedding. Let $j: U=X\setminus Z\to X$ be its complement, and let $\widetilde{j}: V=U\times_X Y\to Y$ be the embedding of its preimage in $Y$.

\begin{equation*}
\begin{xy}
\xymatrix{
Y\times_X Z\ar[r]^{\widetilde{i}}\ar[d]^{\widetilde{f}} & Y\ar[d]^f & V=Y\times_X U\ar[d]^{f'}\ar[l]_{\widetilde{j}}\\
Z\ar[r]_i & X& U \ar[l]^j
}
\end{xy}
\end{equation*}
By assumption, $\M^\bullet$ is $0$-holonomic on $Y$, so we have a distinguished triangle
\begin{equation*}
\widetilde{i}_+\widetilde{i}^!\M^\bullet\to \M^\bullet\to \widetilde{j}_+\widetilde{j}^!\M^\bullet\to
\end{equation*}
of $\C$-complexes on $Y$, and applying $f_+$ yields the distinguished triangle
\begin{equation*}
f_+\widetilde{i}_+\widetilde{i}^!\M^\bullet\to f_+\M^\bullet\to f_+\widetilde{j}_+\widetilde{j}^!\M^\bullet\to.
\end{equation*}
The composition results from Proposition \ref{compandadj} imply that the first term is isomorphic to $i_+\widetilde{f}_+\widetilde{i}^!\M^\bullet$, while the third is as before
\begin{align*}
f_+\widetilde{j}_+\widetilde{j}^!\M^\bullet&\cong j_+f'_+\widetilde{j}^!\M^\bullet\\
&\cong \mathrm{R}j_*(f'_+(\M^\bullet|_V))\\
&\cong \mathrm{R}j_*((f_+\M^\bullet)|_U)\\
&\cong j_+j^!f_+\M^\bullet.
\end{align*}
In particular, since $\widetilde{f}_+\widetilde{i}^!\M^\bullet$ is a $\C$-complex by Corollary \ref{zeroholinv} and Theorem \ref{Ccomplexprops}.(iii), so is $i_+\widetilde{f}_+\widetilde{i}^!\M^\bullet$, and thus $j_+j^!f_+\M^\bullet$ is also a $\C$-complex. These isomorphisms also identify the triangle above with the localisation triangle
\begin{equation*}
i_+i^!f_+\M^\bullet\to f_+\M^\bullet\to j_+j^!f_+\M^\bullet\to
\end{equation*}
from Proposition \ref{triangle}. 

Therefore, $\widetilde{f}_+\widetilde{i}^!\M^\bullet$ and $i^!f_+\M^\bullet$ are two $\C$-complexes on $Z$ which become isomorphic after applying $i_+$, so they are indeed isomorphic by Theorem \ref{Ccomplexprops}.(i).

\end{proof}
\begin{prop}
\label{zeroholproj}
Let $g: X\to Y$ be a projective morphism between smooth rigid analytic spaces. If $\M^\bullet\in \mathrm{D}(\w{\D}_X)$ is $0$-holonomic, then $g_+\M^\bullet$ is $0$-holonomic.
\end{prop}
\begin{proof}
Let $f: Y'\to Y$ be a smooth morphism, and let $U\subseteq Y'$ be the complement of a divisor $Z$. By setting $X'=Y'\times_YX$, $U'=X\times_Y U$, we obtain the following Cartesian squares
\begin{equation*}
\begin{xy}
\xymatrix{
U'\ar[r]^{\widetilde{g}} \ar[d]_{\widetilde{j}} & U\ar[d]^j\\
X'\ar[d]_{\widetilde{f}} \ar[r]^{g'} & Y'\ar[d]^f\\
X\ar[r]_g & Y.
}
\end{xy}
\end{equation*}
Note that by Theorem \ref{Ccomplexprops}.(iii), $g_+\M^\bullet$ is a $\C$-complex on $Y$, and since $\M^\bullet$ is $0$-holonomic, we have in the same way that $\widetilde{j}^!\widetilde{f}^!\M^\bullet$, $\widetilde{g}_+\widetilde{j}^!\widetilde{f}^!\M^\bullet$ and $j^!f^!g_+\M^\bullet$ are $\C$-complexes. It thus follows from Theorem \ref{basechange} that 
\begin{equation*}
j^!f^!g_+\M^\bullet\cong \widetilde{g}_+\widetilde{j}^!\widetilde{f}^!\M^\bullet.
\end{equation*} 
In particular, $j_+j^!f^!g_+\M^\bullet\cong j_+\widetilde{g}_+\widetilde{j}^!\widetilde{f}^!\M^\bullet\cong g'_+\widetilde{j}_+\widetilde{j}^!\widetilde{f}^!\M^\bullet$ by commutativity of the top square. As $g'_+$ preserves $\C$-complexes by Theorem \ref{Ccomplexprops}.(iii), this is a $\C$-complex, since $\widetilde{j}_+\widetilde{j}^!\widetilde{f}^!\M^\bullet\in \mathrm{D}_{\C}(\w{\D}_{X'})$ by virtue of $\M^\bullet$ being $0$-holonomic. 

Having shown that $j_+j^!f^!g_+\M^\bullet\in \mathrm{D}_{\C}(\w{\D}_{Y'})$, Proposition \ref{triangle} implies that $\mathrm{R}\underline{\Gamma}_Z(f^!g_+\M^\bullet)$ is a $\C$-complex as well. Thus $g_+\M^\bullet$ is indeed $0$-holonomic.
\end{proof}

\begin{cor}
Let $i: Z\to X$ be a Zariski closed smooth subvariety and let $\M^\bullet\in \mathrm{D}(\w{\D}_Z)$. Then $\M^\bullet$ is $0$-holonomic if and only if $i_+\M^\bullet$ is $0$-holonomic.
\end{cor}

\begin{proof}
Suppose that $i_+\M^\bullet$ is $0$-holonomic. Then by Proposition \ref{Kashiwaratest}, $\M^\bullet$ is a $\C$-complex, and $\M^\bullet\cong i^!i_+\M^\bullet$ by Theorem \ref{Ccomplexprops}.(i). Hence $\M^\bullet$ is $0$-holonomic by Proposition \ref{pullbackalongclosed}.\\
\\
Conversely, if $\M^\bullet$ is $0$-holonomic, then it follows from Proposition \ref{zeroholproj} that $i_+\M^\bullet$ is $0$-holonomic.
\end{proof}

\begin{cor}
	\label{zerohollocal}
Let $Z$ be a Zariski closed subvariety of a smooth rigid analytic variety $X$, with embedding $i: Z\to X$. Let $j: U\to X$ denote its complement. If $\M^\bullet$ is $0$-holonomic on $X$, then $\mathrm{R}\underline{\Gamma}_Z(\M^\bullet)$ and $j_+j^!\M^\bullet$ are $0$-holonomic.

If $Z$ is smooth, then $\mathrm{R}\underline{\Gamma}_Z(\M^\bullet)\cong i_+i^!\M^\bullet$.  
\end{cor}
\begin{proof}
By Lemma \ref{anyclosedsubvar}, $\mathrm{R}\underline{\Gamma}_Z(\M^\bullet)$ is a $\C$-complex, so $j_+j^!\M^\bullet$ is also a $\C$-complex by Proposition \ref{triangle}. If $Z$ is smooth, then Proposition \ref{triangle} implies also that $\mathrm{R}\underline{\Gamma}_Z(\M^\bullet)\cong i_+i^!\M^\bullet$, which is in particular $0$-holonomic by Corollary \ref{zeroholinv} and Proposition \ref{zeroholproj}. The distinguished triangle then implies that $j_+j^!\M^\bullet$ is also $0$-holonomic.

The Mayer-Vietoris sequence in Proposition \ref{MVforDcap} thus implies that $\mathrm{R}\underline{\Gamma}_Z(\M^\bullet)$ (and thus $j_+j^!\M^\bullet$) is $0$-holonomic whenever $Z$ is an algebraic snc divisor in the sense of \cite[subsection 9.2]{DcapThree}.

In the general case, \cite{Temkin} provides a projective morphism $X'\to X$ of smooth rigid spaces which is an isomorphism over $U$ such that $U$ is the complement of an algebraic snc divisor in $X'$. We thus have a commutative diagram
\begin{equation*}
\begin{xy}
\xymatrix{
U\ar[r]^{\widetilde{j}} \ar[d]_{\mathrm{id}} & X'\ar[d]^f\\
U\ar[r]_j & X.
}
\end{xy}
\end{equation*} 
Now $j_+j^!\M^\bullet\cong f_+\widetilde{j}_+\widetilde{j}^!f^!\M^\bullet$ by Proposition \ref{compandadj}, and this is a $0$-holonomic $\C$-complex by Proposition \ref{zeroholproj}, Corollary \ref{zeroholinv}, and the above argument. Thus $j_+j^!\M^\bullet$ is $0$-holonomic, and by Proposition \ref{triangle}, $\mathrm{R}\underline{\Gamma}_Z(\M^\bullet)$ is also $0$-holonomic.
\end{proof}

\section{Holonomic $\C$-complexes}
\begin{defn}
\label{defnhol}
	Let $X$ be a smooth rigid analytic $K$-variety, and let $n>0$. 
	A $\C$-complex $\M^\bullet\in \mathrm{D}_\C(\w{\D}_X)$ is called \textbf{$n$-holonomic} if it is $(n-1)$-holonomic and for any smooth morphism $f: X'\to X$ and any divisor $Z\subseteq X'$, both $\mathbb{D}f^!\M^\bullet$ and $\mathbb{D}\mathrm{R}\underline{\Gamma}_Z(f^!\M^\bullet)$ are $(n-1)$-holonomic.
	
	A $\C$-complex $\M^\bullet$ is called \textbf{holonomic} if it is $n$-holonomic for all $n\geq 0$.
\end{defn}
By the localisation triangle from Proposition \ref{triangle}, we can also consider $\mathbb{D}j_+j^!f^!\M^\bullet$ instead of $\mathbb{D}\mathrm{R}\underline{\Gamma}_Z(f^!\M^\bullet)$, where $j: U\to X'$ is the complement of $Z$.

We write $\mathrm{D}_{n-\mathrm{hol}}(\w{\D}_X)$ for the full subcategory of $n$-holonomic complexes, and $\mathrm{D}_{\mathrm{hol}}(\w{\D}_X)$ for the full subcategory of holonomic complexes. We say that a coadmissible $\w{\D}_X$-module is holonomic if it is holonomic when viewed as a complex concentrated in degree zero. We warn again that $\mathrm{D}_{\mathrm{hol}}(\w{\D}_X)$ is not the same as the category of $\C$-complexes with holonomic cohomology, as we will discuss in more detail in Lemma \ref{unboundedconnection}.

\begin{lem}
	\label{nholtriang}
	For each $n\geq 0$, $\mathrm{D}_{n-\mathrm{hol}}(\w{\D}_X)$ forms a full triangulated subcategory of $\mathrm{D}(\w{\D}_X)$.
	
	The subcategory $\mathrm{D}_{\mathrm{hol}}(\w{\D}_X)$ is a full triangulated subcategory of $\mathrm{D}(\w{\D}_X)$. 
\end{lem}
\begin{proof}
	This follows straightforwardly by induction, starting with Lemma \ref{zeroholtriang} for the base case $n=0$.
\end{proof}

\begin{lem}
	\label{nholanycl}
	Let $\M^\bullet$ be an $n$-holonomic $\C$-complex on $X$, $n\geq 1$. Then for any smooth morphism $f: X'\to X$ and any closed subvariety $Z\subseteq X'$, the complex $\mathbb{D}\mathrm{R}\underline{\Gamma}_Z(f^!\M^\bullet)$ is $(n-1)$-holonomic. 
	If $j: U\to X'$ is the complement of $Z$, then $\mathbb{D}j_+j^!f^!\M^\bullet$ is $(n-1)$-holonomic.
\end{lem}
\begin{proof}
	Writing $Z$ as an intersection of divisors, the first statement follows from Proposition \ref{MVforDcap} and Lemma \ref{nholtriang}, while the second follows from Proposition \ref{triangle} and Lemma \ref{nholtriang}.
\end{proof}

We now study the stability properties of holonomic complexes. As the definition already is formulated in terms of behaviour under our six functors, it is not surprising that this is mainly formal. For instance, it is immediate from the definition that $\mathbb{D}$ sends $n$-holonomic complexes to $(n-1)$-holonomic complexes, so that the subcategory $\mathrm{D}_{\mathrm{hol}}(\w{\D}_X)$ is stable under the duality operation.

\begin{lem}
	Let $X$ be a smooth rigid analytic $K$-variety and let $\M^\bullet$ be an $n$-holonomic $\C$-complex, $n\geq 0$. If $f: X'\to X$ is a smooth morphism, then $f^!\M^\bullet$ is also $n$-holonomic.
\end{lem}
\begin{proof}
	Let $g: X''\to X'$ be a smooth morphism. Since the composition of smooth morphisms is smooth, it follows from the composition result Proposition \ref{compandadj} that 
	\begin{equation*}
		\mathbb{D}g^!f^!\M^\bullet\cong \mathbb{D}(fg)^!\M^\bullet
	\end{equation*}
	is $(n-1)$-holonomic. Likewise, if $Z\subseteq X''$ is a divisor, then 
	\begin{equation*}
		\mathbb{D}\mathrm{R}\underline{\Gamma}_Z(g^!f^!\M^\bullet)\cong \mathbb{D}\mathrm{R}\underline{\Gamma}_Z((fg)^!\M^\bullet)
	\end{equation*}
	is $(n-1)$-holonomic. Thus, $f^!\M^\bullet$ is $n$-holonomic.
\end{proof}

\begin{prop}
Let $i: Z\to X$ be a Zariski closed smooth subvariety and let $n\geq 0$. If $\M^\bullet$ is $n$-holonomic on $X$, then $i^!\M^\bullet$ is $n$-holonomic on $Z$.
\end{prop}
\begin{proof}
We argue inductively, with the case $n=0$ being done in Proposition \ref{pullbackalongclosed}. 

If $n\geq 1$, assume the statement is true for $n-1$. Let $\M^\bullet\in \mathrm{D}_\C(\w{\D}_X)$ be $n$-holonomic. Let $f: Z'\to Z$ be a smooth morphism and let $j: U\to Z'$ be the complement of some divisor. Arguing locally, we can again assume that $X\cong Z\times \mathbf{B}^d$, so that the morphisms fit into the following commutative diagram
\begin{equation*}
\begin{xy}
\xymatrix{
U\ar[r]^s \ar[d]_j& U\times \mathbf{B}^d\ar[d]^{j'=j\times \mathrm{id}}& Y\ar[l]_{\phi}\\
Z'\ar[d]_f\ar[r]^{i'} & Z'\times \mathbf{B}^d\ar[d]^{f'=f\times \mathrm{id}}\\
Z\ar[r]^i & Z\times \mathbf{B}^d
}
\end{xy}
\end{equation*} 
as in Proposition \ref{pullbackalongclosed}.

Note that $\mathbb{D}f^!i^!\M^\bullet$ is $\C$-complex, since $\M^\bullet$ is $0$-holonomic, and Proposition \ref{Poincare} yields
\begin{equation*}
i'_+\mathbb{D}f^!i^!\M^\bullet\cong \mathbb{D}i'_+i'^!f'^!\M^\bullet,
\end{equation*}
which is $(n-1)$-holonomic by definition. But now
\begin{equation*}
\mathbb{D}f^!i^!\M^\bullet\cong i'^!i'_+\mathbb{D}f^!i^!\M^\bullet
\end{equation*}
by Kashiwara's equivalence, and the latter is $(n-1)$-holonomic by induction hypothesis.

Similarly, it now suffices to show that $i'_+\mathbb{D}j_+j^!f^!i^!\M^\bullet$ is $(n-1)$-holonomic and apply again Theorem \ref{Ccomplexprops}.(i) and the induction hypothesis. Note that
\begin{equation*}
i'_+\mathbb{D}j_+j^!f^!i^!\M^\bullet\cong \mathbb{D}i'_+ j_+j^!f^!i^!\M^\bullet\cong \mathbb{D}j'_+s_+s^!j'^!f'^!\M^\bullet,
\end{equation*}
by Proposition \ref{Poincare} and Proposition \ref{compandadj}. The latter object fits into a distinguished triangle with
\begin{equation*}
\mathbb{D}j'_+j'^!f'^!\M^\bullet
\end{equation*}
and 
\begin{equation*}
\mathbb{D}j'_+\phi_+\phi^!j'^!f'^!\M^\bullet,
\end{equation*}
which are both $(n-1)$-holonomic by $n$-holonomicity of $\M^\bullet$ and Lemma \ref{nholanycl}. This concludes the proof.
\end{proof}

\begin{prop}
\label{nholproj}
	Let $g: X\to Y$ be a projective morphism between smooth rigid analytic $K$-spaces. Let $n\geq 0$. If $\M^\bullet\in \mathrm{D}_\C(\w{\D}_X)$ is $n$-holonomic, then $g_+\M^\bullet$ is $n$-holonomic.
\end{prop}
\begin{proof}
The case $n=0$ is precisely Proposition \ref{zeroholproj}.

Now let $n>0$ and assume that the statement is true for $(n-1)$-holonomic $\C$-complexes. Consider the same diagram
\begin{equation*}
\begin{xy}
\xymatrix{
V'\ar[r]^{\widetilde{g}} \ar[d]_{\widetilde{j}} & V\ar[d]^j\\
X'\ar[r]^{g'} \ar[d]_{\widetilde{f}} & Y'\ar[d]^f\\
X\ar[r]_g & Y
}
\end{xy}
\end{equation*}
as in Proposition \ref{zeroholproj}. Then it follows from Theorem \ref{basechange} and Proposition \ref{Poincare} that 
\begin{equation*}
\mathbb{D}f^!g_+\M^\bullet\cong \mathbb{D}g'_+\widetilde{f}^!\M^\bullet\cong g'_+\mathbb{D}\widetilde{f}^!\M^\bullet
\end{equation*}
and
\begin{align*}
\mathbb{D}j_+j^!f^!g_+\M^\bullet&\cong \mathbb{D}j_+\widetilde{g}_+\widetilde {j}^!\widetilde{f}^!\M^\bullet\\
&\cong \mathbb{D}g'_+\widetilde{j}_+\widetilde{j}^!\widetilde{f}^!\M^\bullet\\
& \cong g'_+\mathbb{D}\widetilde{j}_+\widetilde{j}^!\widetilde{f}^!\M^\bullet.
\end{align*}
As these complexes are $(n-1)$-holonomic by induction hypothesis, this proves that $g_+\M^\bullet$ is $n$-holonomic.
\end{proof}

\begin{cor}
	\label{holKashiwara}
	Let $i: Z\to X$ be a Zariski closed smooth subvariety and let $\M^\bullet\in \mathrm{D}(\w{\D}_Z)$. Let $n\geq 0$. Then $\M^\bullet$ is $n$-holonomic if and only if $i_+\M^\bullet$ is $n$-holonomic.
\end{cor}

\begin{cor}
	\label{holinv}
	Let $f: X\to Y$ be a morphism between smooth rigid analytic $K$-varieties. Then $f^!$ preserves $n$-holonomicity for any $n\geq 0$. In particular, $f^!$ preserves holonomicity.
\end{cor}

\begin{cor}
	Let $Z\subseteq X$ be a Zariski closed subvariety and let $j:U \to X$ denote the complement. If $\M^\bullet\in \mathrm{D}_{\C}(\w{\D}_X)$ is $n$-holonomic, then so are $j_+j^!\M^\bullet$ and $\mathrm{R}\underline{\Gamma}_Z(\M^\bullet)$.
\end{cor}
\begin{proof}
	The case $n=0$ is done in Corollary \ref{zerohollocal}. In particular, if $Z$ is smooth, then $\mathrm{R}\underline{\Gamma}_Z(\M^\bullet)\cong i_+i^!\M^\bullet$. Therefore, Corollary \ref{holinv} and Corollary \ref{holKashiwara} imply that $\mathrm{R}\underline{\Gamma}_Z(\M^\bullet)$ is $n$-holonomic whenever $Z$ is smooth, and more generally when $Z$ is an algebraic snc divisor by Proposition \ref{MVforDcap}. By Proposition \ref{triangle}, we then also have $j_+j^!\M^\bullet\in \mathrm{D}_{n-\mathrm{hol}}(\w{\D}_X)$ in that case.
	
	The general case follows as before from a desingularization argument, using Corollary \ref{holinv}, Proposition \ref{compandadj} and Proposition \ref{nholproj}.
\end{proof}

Putting everything together, we have the following:
\begin{thm}
	\label{sixffhol}
	Let $X$, $Y$ be smooth rigid analytic $K$-varieties, and let $f: X\to Y$ be a morphism.
	\begin{enumerate}[(i)]
		\item The duality functor $\mathbb{D}$ sends $\mathrm{D}_{\mathrm{hol}}(\w{\D}_X)$ to $\mathrm{D}_{\mathrm{hol}}(\w{\D}_X)^{\mathrm{op}}$.
		\item The inverse image and extraordinary inverse image functors $f^+$ resp. $f^!$ send $\mathrm{D}_{\mathrm{hol}}(\w{\D}_Y)$ to $\mathrm{D}_{\mathrm{hol}}(\w{\D}_X)$. 
		\item If $f$ is projective, then $f_+$ and $f_!$ send $\mathrm{D}_{\mathrm{hol}}(\w{\D}_X)$ to $\mathrm{D}_{\mathrm{hol}}(\w{\D}_Y)$. (By Proposition \ref{Poincare}, $f_+\cong f_!$ in this case.)
		\item If $Z\subseteq X$ is a closed subvariety with complement $j: U\to X$, then $\mathrm{R}\underline{\Gamma}_Z$ and $j_+j^!$ send $\mathrm{D}_{\mathrm{hol}}(\w{\D}_X)$ to $\mathrm{D}_{\mathrm{hol}}(\w{\D}_X)$.
	\end{enumerate}
\end{thm}

In this way, we obtain stability under five of the six operations: duality, both inverse image functors, and both direct image functors if the morphism is projective. We will discuss general issues with the direct image operations at the end of the section, but remark that the operation $j_+j^!$ above is at least a partial remedy. Notably, the sixth functor, tensor product, has not yet made an appearance.

We also have the base change and adjunction results given in Theorem \ref{basechange} and Proposition \ref{compandadj}.

Defining holonomicity in terms of the behaviour under our functors has made these stability properties essentially formal, but the price which we pay is that we do not have an explicit characterization of holonomic complexes. A priori, it is not even clear whether $\mathrm{D}_{\C}(\w{\D}_X)$ is non-zero. We now show that it contains in fact all integrable connections, as one would hope.

For this, we note that instead of divisors, we can also test against smooth closed subvarieties.
\begin{lem}
\label{smoothtest}
Let $\M^\bullet$ be an $(n-1)$-holonomic $\C$-complex on a smooth rigid analytic $K$-variety $X$, $n\geq 1$. If for \textbf{all} morphisms $f: X'\to X$ and smooth Zariski closed subvarieties $Z\subset X'$, both $\mathbb{D}f^!\M^\bullet$ and $\mathbb{D}\mathrm{R}\underline{\Gamma}_Zf^!\M^\bullet$ are $(n-1)$-holonomic, then $\M^\bullet$ is $n$-holonomic.
\end{lem}
\begin{proof}
Let $f: X'\to X$ be smooth and let $Z\subseteq X'$ be a divisor with complement $j: U\to X'$. If $Z$ is smooth, then $\mathbb{D}f^!\M^\bullet$ and $\mathbb{D}\mathrm{R}\underline{\Gamma}_Z(f^!\M)$ are $(n-1)$-holonomic by assumption, and a Mayer-Vietoris argument yields the same result in the case when $Z$ is an algebraic snc divisor.

In the general case, consider again the resolution of singularities
\begin{equation*}
\begin{xy}
\xymatrix{
U\ar[r]^{\widetilde{j}} \ar[d]_{\mathrm{id}}& X''\ar[d]^\rho\\
U\ar[r]_j & X'
}
\end{xy}
\end{equation*}
where $U\to X''$ is the complement of an snc divisor $\widetilde{Z}$, and $\rho$ is projective.

We remark that in this case, both $\mathbb{D}f^!\M^\bullet$ and $\mathbb{D}\rho^!f^!\M$ are $(n-1)$-holonomic by assumption.

We wish to show that $\mathbb{D}j_+j^!f^!\M$ is $(n-1)$-holonomic. Note that by the above diagram and Proposition \ref{Poincare},
\begin{align*}
\mathbb{D}j_+j^!f^!\M^\bullet&\cong \mathbb{D}\rho_+\widetilde{j}_+\widetilde{j}^!\rho^!f^!\M^\bullet\\
&\cong \rho_+\mathbb{D}\widetilde{j}_+\widetilde{j}^!\rho^!f^!\M^\bullet,
\end{align*}
so by Proposition \ref{nholproj} it suffices to show that $\mathbb{D}\widetilde{j}_+\widetilde{j}^!\rho^!f^!\M$ is $(n-1)$-holonomic. But since $\mathbb{D}\mathrm{R}\underline{\Gamma}_{\widetilde{Z}}(\rho^!f^!\M^\bullet)$ and $\mathbb{D}\rho^!f^!\M^\bullet$ are $(n-1)$-holonomic by assumption, this follows immediately from Lemma \ref{nholtriang}.

Thus $\mathbb{D}j_+j^!f^!\M^\bullet$ is $(n-1)$-holonomic, so that $\mathbb{D}\mathrm{R}\underline{\Gamma}_Z(f^!\M^\bullet)$ is $(n-1)$-holonomic by Lemma \ref{nholtriang}. Therefore, $\M^\bullet$ is indeed $n$-holonomic.
\end{proof}

\begin{thm}
	\label{ICishol}
If $\M$ is an integrable connection on a smooth rigid analytic variety $X$, then $\M$ is holonomic.
\end{thm}
\begin{proof}
By Proposition \ref{ICzero}, any integrable connection is $0$-holonomic. We proceed by induction. Let $n>0$ and suppose that any integrable connection is $(n-1)$-holonomic. If $\M$ is an integrable connection on $X$ and $f: X'\to X$ is any morphism of smooth rigid analytic $K$-varieties, then $\mathbb{D}f^!\M$ is still an integrable connection by \cite[Proposition 7.3]{Bode6Op}, shifted in degree, and hence $(n-1)$-holonomic. If $i: Z\to X'$ is a smooth closed subvariety, then (since $f^!\M$ is $0$-holonomic) $\mathrm{R}\underline{\Gamma}_Z(f^!\M)\cong i_+i^!f^!\M$. By Lemma \ref{smoothtest}, we need to show that $\mathbb{D}i_+i^!f^!\M$ is $(n-1)$-holonomic. But 
\begin{equation*}
\mathbb{D}i_+i^!f^!\M\cong i_+\mathbb{D}i^!f^!\M,
\end{equation*} 
and $\mathbb{D}i^!f^!\M$ is an integrable connection up to degree shift by \cite[Proposition 7.3]{Bode6Op} again, in particular $(n-1)$-holonomic by assumption. We can thus finish the proof by invoking Proposition \ref{nholproj}.
\end{proof}

We make one additional remark about the direct image functor $f_+$. In the case of the structure morphism to a point, $f_+(\O_X)$ computes de Rham cohomology by \cite[Proposition 9.13]{Bode6Op}. In particular, the pathologies of de Rham cohomology in the non-proper case make it clear that holonomicity is not stable under $f_+$ for arbitrary $f$. While one might hope to obtain slightly stronger results by considering an `overconvergent' version, there are several other examples demonstrating the problems with $f_+$:
\begin{enumerate}[(i)]
	\item If $f: X=\Sp K\langle x\rangle\to \Sp K\langle \pi x\rangle=Y$ is the embedding of one closed disc inside a larger disc, it is clear that $f_+\O_X=f_*\O_X$ cannot be a coadmissible $\w{\D}_Y$-module (for example, it contains $\O_Y$ as a dense submodule). In particular, there are partially proper morphisms $f$ such that $f_+\O$ is not a $\C$-complex.
	\item Considering a sum of skyscrapers on points $\{x=p^{-n}\}\subseteq \mathbb{A}^{1, \mathrm{an}}_K$ (or something similar on an open disc) yields a holonomic $\w{\D}$-module with infinite-dimensional de Rham cohomology. Once again, it seems unlikely that this can be addressed with overconvergent variants.
	\item The examples in \cite{Bitoun} describe integrable connections whose pushforward along a Zariski open embedding is not a $\C$-complex. 
\end{enumerate} 

We expect however that holonomicity is preserved under $f_+$ for any \emph{proper} morphism $f$.
\section{Further examples and questions}
In this section, we compare the notions of weak holonomicity, $0$-holonomicity and holonomicity through examples. We also indicate some open problems.

Recall that \cite{DcapThree} defines a coadmissible $\w{\D}_X$-module $\M$ to be \textbf{weakly holonomic} if $\mathbb{D}\M$ is concentrated in degree zero. We begin by recalling from \cite[subsection 8.1]{DcapThree} an example of a weakly holonomic $\w{\D}$-module which is not $0$-holonomic.

\begin{lem}
	There exist weakly holonomic modules which are not $0$-holonomic.
\end{lem}
\begin{proof}
	Let $X=\Sp K\langle x\rangle$, and let $\partial=\frac{\mathrm{d}}{\mathrm{d}x}$. Then $\theta=\prod_{i=0}^\infty (1-\pi^i\partial)$ defines an element of $\w{\D}_X(X)$, and $\M:=\w{\D}_X/\w{\D}_X\cdot \theta$ is a coadmissible $\w{\D}_X$-module, since finitely presented modules are coadmissible. Let $\A=R\langle x\rangle$, $\pi^n\L=\A\cdot \pi^n\partial$, and let $\D_n$ denote the corresponding sheaf. Then $\D_n\widetilde{\otimes}_{\w{\D}_X} \M\cong \D_n/\D_n\cdot \prod_{i=0}^n (1-\pi^i \partial)$, as the remaining factors in the product are units in $\D_n$. This is a vector bundle of rank $n+1$, so it follows that $\M$ is weakly holonomic. On the other hand, we have a surjection $\M\to \w{\D}_X/\w{\D}_X\cdot \prod_{i=0}^n (1-\pi^i\partial)$ for each $n$. Therefore, the fibre of $\M$ at any point of $X$ is infinite-dimensional. This implies however that for a point $i: x\to X$, $i^!\M$ is not a $\C$-complex. Since $0$-holonomicity is stable under arbitrary extraordinary inverse images, this shows that $\M$ is not $0$-holonomic. 
\end{proof}

This is reassuring in some sense. One main motivation for our definitions was the desire to exclude examples as the above, which have many additional undesirable qualities, e.g. $\M$ above is not of finite length. The following example is more subtle.

\begin{lem}
	\label{typelift}
Let $X=\Sp K\langle x\rangle$ be the closed unit disc and let $j: U=X\setminus \{0\}\to X$. There exists a coherent algebraic $\D_X$-module $\M$ of minimal dimension in the sense of \cite{Mebkhout} with the following properties:
\begin{enumerate}[(i)]
\item $\w{\M}:=\w{\D}_X{\otimes}_{\D_X} \M$ is a coadmissible, weakly holonomic $\w{\D}_X$-module.
\item $\w{\M}|_U$ is an integrable connection, and $\w{\M}\cong j_+j^!\w{\M}$ via the natural morphism.
\item $\w{\M}$ is not $0$-holonomic.
\end{enumerate} 
\end{lem}  
\begin{proof}
Let $\lambda\in K$ be a scalar of positive type (cf. \cite[Definition 3.1]{Bitoun}) such that $p\lambda$ is of type zero. Such a $\lambda$ exists, e.g. by taking a scalar of type zero and repeatedly multiplying with $p^{-1}$, noting that any $\mu\in K$ with $|\mu|>1$ has type $1$. Now set $\M=\D_X/\D_X\cdot(x\partial-\lambda)$, so that $\w{\M}=\w{\D}_X/\w{\D}_X\cdot(x\partial-\lambda)$.

Then $\M$ is of minimal dimension, since $\mathcal{H}om_{\D_X}(\M, \D_X)=0$, and $\w{\M}$ is coadmissible and weakly holonomic by \cite[Proposition 7.2]{DcapThree}. Moreover, $\w{\M}|_U$ is the integrable connection $\O_Ux^\lambda$ from \cite[section 5]{Bitoun}, so that (ii) follows from \cite[Theorem 1.3, Theorem 5.2]{Bitoun}.

To show (iii), consider the morphism $f: X'=\Sp K\langle x^{1/p}\rangle \to X$. Let $j':U'=X'\setminus \{U\}\to X'$, $f_U: U'\to U$. If $\w{\M}$ was $0$-holonomic, then $j'_+j'^!(f^!\w{\M})$ would be a $\C$-complex, and by Proposition \ref{compandadj}, this would be isomorphic to
\begin{equation*}
j'_+(f_U^! \w{\M}|_U)\cong j'_*(\O_{U'} (x^{1/p})^{p\lambda}).
\end{equation*} 
But since $p \lambda$ is not of positive type, \cite[Theorem 1.3]{Bitoun} implies that this is not coadmissible.
\end{proof}

This leads to the following questions.
\begin{ques}
Let $X=\Sp K\langle x\rangle$ and let $\lambda\in K$. Let $\M_\lambda=\w{\D}_X/\w{\D}_X\cdot(x\partial-\lambda)$. For which $\lambda$ is $\M_\lambda$ a $0$-holonomic $\w{\D}_X$-module? For which $\lambda$ is it holonomic?
\end{ques}

\begin{ques}
	\label{holvsweakly}
Is every holonomic $\w{\D}$-module weakly holonomic? Is every $0$-holonomic $\w{\D}$-module weakly holonomic?
\end{ques}

Question \ref{holvsweakly} is also related to the following:
\begin{lem}
	\label{holcohom}
	Suppose that every $0$-holonomic $\w{\D}$-module is weakly holonomic. Then a bounded $\C$-complex $\M^\bullet\in \mathrm{D}^b_{\C}(\w{\D}_X)$ on a smooth rigid analytic $K$-variety $X$ is holonomic if and only if $\mathrm{H}^q(\M^\bullet)$ is holonomic for each $q$.
\end{lem}
\begin{proof}
	By Corollary \ref{zeroholcohom}, a bounded $\C$-complex is $0$-holonomic if and only if its cohomology groups are $0$-holonomic.
	
	If $f: X'\to X$ is a smooth morphism and $Z\subseteq X'$ is a divisor with complement $j: U\to X'$, then $f^!, j^!$ and $j_+$ are exact on coadmissible modules, and $\mathbb{D}$ is exact on weakly holonomic modules.
	
	Hence, if $\M^\bullet\in \mathrm{D}_{\C}^b(\w{\D}_X)$ such that $\mathrm{H}^q(\M^\bullet)$ is $0$-holonomic for each $q$, then
	\begin{equation*}
		\mathrm{H}^q(\mathbb{D}j_+j^!f^!\M^\bullet)\cong \mathbb{D}j_+j^!f^!(\mathrm{H}^{q+\mathrm{dim}X'-\mathrm{dim}X}(\M^\bullet))\in \C_{X'}.
	\end{equation*}
	In particular, $\mathbb{D}f^!\M^\bullet$ and $\mathrm{D}j_+j^!f^!\M^\bullet$ are $\C$-complexes by \cite[Corollary 8.7]{Bode6Op}.
	
	The description of cohomology groups above and an inductive argument now allows us to prove that bounded $\C$-complexes are $n$-holonomic if and only if their cohomology groups are $n$-holonomic, and the result follows.
\end{proof}

The boundedness assumption is necessary, however:

\begin{lem}
	\label{unboundedconnection}
	Let $X=\Sp K\langle x\rangle$. There exists a $\C$-complex $\M^\bullet$ on $X$ such that $\mathrm{H}^q(\M^\bullet)$ is an integrable connection for each $q$, but $\M^\bullet$ is not $0$-holonomic.
\end{lem}
\begin{proof}
	We use the same notation as in the beginning of the section.
	
	For $n\geq 0$, let $\M^n:=\w{\D}_X/\w{\D}_X\cdot (\pi^n\partial-1)$. For $n<0$, let $\M^n=0$, and consider the chain complex $\M^\bullet$, with each morphism the zero map. If $n\geq 0$, then $\M^n$ is an integrable connection of rank $1$ such that $\D_m\widetilde{\otimes}_{\w{\D}_X}\M^n=0$ for $m<n$. In particular, $\M^\bullet$ is a $\C$-complex. Since pulling back to the origin yields an unbounded complex and not an object in $\mathrm{D}_{\C}(\Sp K)=\mathrm{D}_{\mathrm{f.d.}}^b(\mathrm{Vect}_K)$, $\M^\bullet$ is not $0$-holonomic.
\end{proof}

\begin{ques}
	Is every holonomic complex on a smooth affinoid bounded?
\end{ques}

As in the theory of arithmetic $\mathscr{D}$-modules, we expect to obtain a more explicit characterization of holonomic modules as those $\C$-complexes which are constructible out of integrable connections in a suitable sense. The example in Lemma \ref{typelift} already illustrates that the correct notion of constructibility will have to encode some additional `regularity' condition at the boundary of strata.

Inspired by this heuristic, we ask the following.
\begin{ques}
	\label{consques}
	\leavevmode
	\begin{enumerate}[(i)]
		\item Is there a notion of `constructibility' such that a $\C$-complex is holonomic if and only if it is constructible out of integrable connections?
		\item Is every holonomic $\w{\D}$-module of finite length?
	\end{enumerate}
\end{ques}

We remark that the question of constructibility is also crucial for the final of our six operations, tensor products. In \cite{CaroTsuzuki}, Caro--Tsuzuki proved stability of overholonomicity under tensor products by showing that the notion of overholonomicitiy is equivalent to being constructible out of overconvergent $F$-isocrystals in a suitable sense, and then using a d\'evissage argument. One may hope that a similar approach can be implemented in our setting.

\end{document}